\documentclass[11pt]{amsart}
\usepackage{amsmath}
\usepackage{amsthm}
\usepackage{amscd}
\usepackage{amssymb}
\usepackage{amsfonts}
\usepackage{graphics}
\usepackage{color}
\usepackage{comment}
\usepackage{verbatim}
\usepackage{enumitem}
\usepackage[dvipsnames]{xcolor}
\usepackage[all,cmtip]{xy}
\usepackage[normalem]{ulem}
\usepackage{cancel}

\usepackage{hyperref}

\theoremstyle{plain}
\newtheorem{theorem}{Theorem}[section]
\newtheorem{proposition}[theorem]{Proposition}
\newtheorem{lemma}[theorem]{Lemma}

\newtheorem{corollary}[theorem]{Corollary}
\newtheorem{remark}[theorem]{Remark}

\newcommand{\wmns}{ W_{\mu_{N,s}}(\mathbb{R}^N)}
\newcommand{\W}{W_+}

\newcommand{\mG}{\mathcal G}

\newcommand{\J}{\mathcal{J}_{N,s,T}}

\newcommand{\intrn}{\int_{\mathbb{R}^N}}

\newcommand{\dis}{\displaystyle}
\newcommand{\R}{\mathbb{R}}

\theoremstyle{definition}
\newtheorem{definition}[theorem]{Definition}

\title[Extremals for singular Onofri inequality in weighted Sobolev spaces]{Extremals for singular Onofri type inequality in weighted Sobolev spaces}

\author[N. Borgia]{Natalino Borgia}
\address{Dipartimento di  Matematica  \\ Universit\`{a} degli Studi di Bari Aldo Moro \\ Via Orabona 4\\ 70125 Bari, Italy}
\email{natalino.borgia@uniba.it}
\author[S. Cingolani]{Silvia Cingolani}
\address{Dipartimento di  Matematica  \\ Universit\`{a} degli Studi di Bari Aldo Moro \\ Via Orabona 4\\ 70125 Bari, Italy}
\email{silvia.cingolani@uniba.it}
\author{Rub\'en Fi\~nana}
\address{%
	Department of Mathematics\\ University of Almer\'ia\\
	Ctra. Sacramento s/n, La Ca\~nada de San Urbano\\
	Almer\'ia 04120\\
	Spain}
\email{rfa803@ual.es}
\author[G. Mancini]{Gabriele Mancini}
\address{Dipartimento di  Matematica  \\ Universit\`{a} degli Studi di Bari Aldo Moro \\ Via Orabona 4\\ 70125 Bari, Italy}
\email{gabriele.mancini@uniba.it}

\begin{document}
	\begin{abstract}
		The $N$-dimensional Euclidean Onofri inequality was proved by Del Pino and Dolbeault \cite{dPinDol-2013} for smooth compactly supported functions in $\mathbb{R}^N$ with $N \ge 2$. In this paper, we consider the singular version introduced in \cite{LamLu-2018}, and extend it to a natural weighted Sobolev space. Next, we establish a sharp singular Carleson-Chang type inequality and prove its equivalence with the singular $N$-dimensional Euclidean Onofri inequality. Furthermore, we prove the existence of nontrivial, non-constant optimizers for the functional associated with the singular Euclidean Onofri inequality, answering open questions left by Del Pino and Dolbeault in \cite{dPinDol-2013} in the non-singular case, and by Lam and Lu in \cite{LamLu-2018} in the case with singularities. To this end, we exploit the classification result obtained in \cite{CirEspLi-2026} for a singular Liouville-type equation involving the $N$-Laplacian. We also prove the conformal invariance of the functional, establish the sharpness of the singular $N$-dimensional Euclidean Onofri inequality in terms of optimal constants, and state several consequences, including a sharp Poincaré inequality.

	\end{abstract}
	
	
	
	\date{\today}

	\maketitle

	\tableofcontents
	
	\section{Introduction}

	In \cite{dPinDol-2013} Del Pino and Dolbeault obtained the following $N$-dimensional Euclidean Onofri inequality for any $u \in C_0^{\infty}(\mathbb{R}^N)$:
	\begin{equation}\label{EuclideanOnofriNdimensional}
		\displaystyle  \ln \left( \int_{\mathbb{R}^N} e^u \, d\mu_N \right) \leq  \frac{1}{\widetilde{\omega_N}} \int_{\mathbb{R}^N} H_N(u) \, dx + \int_{\mathbb{R}^N} u \, d\mu_N,
	\end{equation}
	where, denoting by $\omega_{N-1}$ and $V_{N}$ the $(N-1)$-dimensional measure of the unit sphere $\mathbb{S}^{N-1}$ and the volume of the unit ball in $\mathbb{R}^N$ respectively, we use the notation
	\begin{equation*}\label{Not2} \mu_N(x):= \frac{1}{V_N \left( 1+ |x|^{\frac{N}{N-1}} \right)^N},\quad  d \mu_N(x):=\mu_N(x) \, dx,\quad \widetilde{\omega_N}:= N^N \left( \frac{N}{N-1} \right)^{N-1} \omega_{N-1},
	\end{equation*} 
	$$H_N(u) :=R_N(\nabla v_N,\nabla u), \quad \text{with} \quad v_N(x):= \ln \mu_N(x), $$
	 and 
	 $$ \displaystyle R_N(X,Y):=|X+Y|^N - |X|^N  -N|X|^{N-2}X \cdot Y \qquad (X,Y) \in \mathbb{R}^N \times \mathbb{R}^N.$$
	
	\medskip
	In dimension $N=2$, by means of the stereographic projection, inequality \eqref{EuclideanOnofriNdimensional} is equivalent to the classical Onofri inequality, proved in \cite{Ono-1982} for  functions $v$ in the Sobolev space  $H^1(\mathbb{S}^2)$: 
	\begin{equation}\label{Onofri}
		\ln\left(\int_{\mathbb{S}^2} e^v\, d\theta\right)\leq \frac{1}{4}\int_{\mathbb{S}^2}|\nabla v|^2\, d\theta+\int_{\mathbb{S}^2} v \, d\theta, 
	\end{equation}
	where $d\theta$ denotes the standard normalized Lebesgue’s measure on $\mathbb{S}^2$. Inequality \eqref{Onofri} improves the previous Moser-Trudinger inequality, see \cite{Mos-1971,Mos-1973,Tru} and  was later extended by Beckner \cite{Bec-1993} to $\mathbb{S}^N$ for $N\geq 2$. See \cite{BCMsurvey}, along with the references cited there for an overview of the main results involving \eqref{Onofri}. In particular, Onofri inequality  can be obtained via several techniques, including various limiting procedures, mass transportation methods in the radial case, and rigidity methods associated with a related nonlinear flow (we refer to the papers by Dolbeault, Esteban, Jankowiak, and Tarantello \cite{DolEstJan-2015, DET}, and references therein for full proofs and further details). The Onofri inequality also has important applications; for instance, in chemotaxis it appears in the study of the Keller-Segel model (see \cite{CALVCORR, GAJZAK}).

	\medskip
	\noindent
	
	A singular $N$-dimensional Euclidean Onofri inequality was obtained in \cite{LamLu-2018} by Lam and Lu for any $N \geq 2$ and $s \in [0,N)$. Their result can be equivalently formulated for any $u \in C_0^{\infty}(\mathbb{R}^N)$ as
	\begin{equation}\label{OnofriSingularDimensionN}
		\ln\left( \int_{\mathbb{R}^N} e^u \, d\mu_{N,s} \right)\leq \frac{1}{\widetilde{\omega_{N,s}}}\int_{\mathbb{R}^N} H_{N,s}(u)\, dx + \int_{\mathbb{R}^N} u \, d\mu_{N,s},
	\end{equation}
	
	\noindent
	where
	
	$$\dis  \widetilde{\omega_{N,s}} :=\left(\frac{N-s}{N}\right)^{N-1} \widetilde{\omega_{N}}=N^N \left( \frac{N-s}{N-1} \right)^{N-1}\omega_{N-1},$$
	\begin{equation}\label{DefMuns}
		\mu_{N,s}(x):=\frac{N-s}{\omega_{N-1}}\frac{1}{(1+|x|^{\frac{N-s}{N-1}})^N|x|^s},  \qquad g_{N,s}(x):= \ln \left[ \;  |x|^s \mu_{N,s}(x)  \; \right],
	\end{equation}
	and 
	\begin{equation}\label{DefHns}
		H_{N,s}(u):=R_N(\nabla g_{N,s},\nabla u ).
	\end{equation}
	
	For $N=2$ and $s\in [0,2)$, the inequality \eqref{OnofriSingularDimensionN} was established in \cite{DET} (see also \cite{DonLamLu-2018}) for any $u\in L^1(\mathbb{R}^2,d\mu_{2,s})$ satisfying $\vert{}\nabla u\vert{}\in L^2(\mathbb{R}^2,dx)$. Therefore, as a first step in this paper, we identify a natural functional setting where \eqref{OnofriSingularDimensionN} holds for any $N \geq 2$ and $s \in [0,N)$. Precisely, we introduce  the following weighted Sobolev space:
\medskip
\begin{align*}
W_{\mu_{N,s}}(\mathbb{R}^N):=\{ \, u\in L^1(\mathbb{R}^N,d\mu_{N,s}): \; \;  |\nabla u|\in L^{N}(\mathbb{R}^N,dx); \; \;  |\nabla u|^2|\nabla g_{N,s}|^{N-2}\in L^1(\mathbb{R}^N,dx)\},
\end{align*}
	
	\medskip
	\noindent
	equipped with norm 
	\begin{equation*}
		\|u\|_{\mu_{N,s}}:= \int_{\mathbb{R}^N} |u| d{\mu_{N,s}} + \|\nabla u\|_N + \left(\int_{\mathbb{R}^N}|\nabla u|^2 |\nabla g_{N,s}|^{N-2} \, dx\right)^{\frac{1}{2}}.
	\end{equation*}
	
	\noindent
	We first show that smooth, compactly supported functions on $\mathbb{R}^N$ are dense in $W_{\mu_{N,s}}(\mathbb{R}^N)$ with respect to the norm $\Vert{} \cdot \Vert{}_{\mu_{N,s}}$, so that we can extend inequality \eqref{OnofriSingularDimensionN} to $\wmns$:

	\begin{theorem}\label{densitystatementANDextension}	
		Assume $N \geq 2$ and $s\in [0,N)$. Then
		$$ W_{\mu_{N,s}}(\mathbb{R}^N)=\overline{C^{\infty}_0(\mathbb{R}^N)}^{\|\cdot\|_{\mu_{N,s}}}. $$
		Moreover, for any function $u\in W_{\mu_{N,s}}(\mathbb{R}^N)$, it holds that
		\begin{equation}\label{MT-SingularityNewSpace}
			\displaystyle \ln\left( \int_{\mathbb{R}^N} e^u \, d\mu_{N,s} \right)\leq \frac{1}{\widetilde{\omega_{N,s}}}\int_{\mathbb{R}^N} H_{N,s}(u)\, dx + \int_{\mathbb{R}^N} u \, d\mu_{N,s}.
		\end{equation}	
	\end{theorem}	
	
	We observe that the previous theorem recovers the results obtained in \cite{BorCinMan-2025} in the non-singular case, namely for any $N \geq 2$ and $s=0$. Inspired by \cite{BorCinMan-2025}, another step of this paper is to establish a natural connection between inequality \eqref{MT-SingularityNewSpace} and  a singular Carleson-Chang type inequality. To this aim, we recall that Sandeep and Adimurthi in  \cite{SanAdi-2007} proved that if $\Omega$  is a bounded domain in $\mathbb{R}^N$ containing the origin with $N \geq 2$ and $ s \in [0,N)$, then there exists a constant $C_{N,s}>0$ such that
	\begin{equation}\label{MTSingularity}
		\sup_{ \substack{u\in W^{1,N}_0(\Omega) \\ \|\nabla u\|_{L^N(\Omega,dx)}\leq 1}} \frac{1}{|\Omega|^{1-\frac{s}{N}}} \int_\Omega  \frac{e^{\alpha_N (1-\frac{s}{N}) |u|^{\frac{N}{N-1}} }}{|x|^s} \, dx \leq C_{N,s},
	\end{equation}
	where $\alpha_{N}:=N\omega_{N-1}^\frac{1}{N-1}$ is the standard Moser-Trudinger exponent. Moreover, if $\Omega= B_r$, the supremum is attained for any $N \geq 2$ (see \cite{IulManRoy-2017} for $N=2$, and  \cite{CsaNguRoy-2021} for  any $N\geq 2$). As a straightforward consequence of \eqref{MTSingularity} and Young's inequality, one can prove that for every $u\in W^{1,N}_0(\Omega)$,
	\begin{equation*}
		\ln\left(\frac{1}{|\Omega|^{1-\frac{s}{N}}} \int_{\Omega} \frac{e^u }{|x|^s} \, dx \right) \leq \frac{1}{\widetilde{\omega_{N,s}}}\int_\Omega |\nabla u|^N + \ln C_{N,s}.
	\end{equation*}

	
	
	Denoting by $B_1$ the open unit ball of $\mathbb{R}^N$,  we prove that for any $u \in W_0^{1,N}(B_1)$ and $s \in [0,N)$ we have
	\begin{align}\label{Chang-Carleson-Singularity}
		\ln \left(\frac{1}{|B_1|^{1-\frac{s}{N}}} \int_{B_1} \frac{e^u }{|x|^s} \, dx \right)
		< \frac{1}{\widetilde{\omega_{N,s}}}\int_{B_1}| \nabla u|^N \, dx 
		+ \sum_{k=1}^{N-1}\frac{1}{k}  + \ln\left(\frac{N}{N-s} \right)+ \frac{s}{N} \ln V_N. 
	\end{align} 
	
	\noindent
	We notice that for $s=0$ we obtain the classical non-singular Carleson-Chang inequality, proved in  \cite{CarCha-1986}.
	Then, we consider the functional $J_s: W^{1,N}_0(B_1) \to \mathbb{R}$, associated to \eqref{Chang-Carleson-Singularity}, defined as	
	\[J_s(u):=\frac{1}{\widetilde{\omega_{N,s}}}\int_{B_1}|\nabla u|^N \, dx -\ln \left(\frac{1}{|B_1|^{1-\frac{s}{N}}} \int_{B_1} \frac{e^u }{|x|^s} \, dx \right),\]
	showing that
	\[\inf_{W^{1,N}_0(B_1)}{J_s}= -\sum_{k=1}^{N-1}\frac{1}{k} - \ln\left(\frac{N}{N-s}\right)  - \frac{s}{N} \ln V_N .\]
	
	\noindent
	On the other hand, the functional $I_{s} : W_{\mu_{N,s}}(\mathbb{R}^N) \to \mathbb{R}$ associated with \eqref{MT-SingularityNewSpace} is defined by
	\[I_{s}(u):=\frac{1}{\widetilde{\omega_{N,s}}}\int_{\mathbb{R}^N} H_{N,s}(u)\, dx + \int_{\mathbb{R}^N} u \, d\mu_{N,s}- \ln\left( \int_{\mathbb{R}^N} e^u \,  d\mu_{N,s} \right).\]
	
	\noindent
	By  \eqref{MT-SingularityNewSpace} and considering that $I_s(0)=0$, we know that
	$$
	\min_{\wmns} I_s = 0. 
	$$
	
	\noindent
	Thus we obtain an equivalence between  \eqref{MT-SingularityNewSpace} and \eqref{Chang-Carleson-Singularity} in the following sense:
	\begin{theorem}\label{Equivalence} For $N\geq 2$ and $s\in [0, N)$, we have 
		\[\min_{W_{\mu_{N,s}}(\mathbb{R}^N)} I_s= \inf_{W^{1,N}_0(B_1)} J_s + \sum_{k=1}^{N-1} \frac{1}{k}+ \ln\left(\frac{N}{N-s} \right) + \frac{s}{N} \ln V_N  .\]
	\end{theorem}


\noindent
We notice that the weighted Sobolev space $W_{\mu_{N,s}}(\mathbb{R}^N)$ introduced here naturally contains constant functions as trivial minimizers for $I_s$, which we prove to be Fréchet differentiable. 

\smallskip

In this paper we also establish the existence of non-trivial, non-constant minimizers for the functional $I_s$, fully answering a problem left open not only in the singular case (see Lam and Lu in \cite{LamLu-2018}), but also in the case without singularities (see Del Pino and Dolbeault in \cite{dPinDol-2013}).  In addition, we prove that such minimizers are the only critical points.
	These results are strictly related to the conformal invariance of $I_s$. For any $x_0\in \R^N$ and $t>0$, we let $T_{t,x_0}:\R^N\longrightarrow \R^N$ denote the map  
	\begin{equation*}
		T_{t,x_0}(x) := t (x-x_0),
	\end{equation*}
	and we consider the group of transformations
	\begin{equation}\label{G}
		\mG :=\begin{cases}
			\{ T_{t,x_0}\;:\; x_0\in \R^N, t>0\} & \text{ if } s= 0,\\
			\{ T_{t,0}\;:\;  t>0\} & \text{ if } s>0. 
		\end{cases}
	\end{equation}
	For any $T\in \mG$, we let $J_T$ denote the Jacobian matrix of $T$ and we define 
	\begin{equation}\label{notation}
		[T]:= \sqrt[N]{\det J_T} \quad \text{ and } \quad h_T(x):= g_{N,s}(Tx)-g_{N,s}(x) + (N-s) \ln [T].
	\end{equation}
	
	\medskip
	\noindent
	Then, we prove that the functions $h_T$ so defined are minimum points and, up to a constant, the unique critical points of $I_s$, and that the functional $I_s$ is conformally invariant:
	
	\begin{theorem}\label{Invariance-MinimumIs}
		For any $T\in \mG$, we have that $h_T\in \wmns$. Moreover,
		
		\begin{enumerate}[label=$\alph*$), ref=$\alph*$)]
			\item for any $T\in \mG$, $h_T$ is a minimum point for $I_s$, namely $I_s(h_T)=0$;\label{minimumht} \medskip
			\item  all critical points of $I_s$ in $\wmns$ are of the form $h_T+c$ for some $T\in \mG$ and $c\in \R$.\label{allcritical} 
		\end{enumerate}
		
		\smallskip
		Furthermore, $I_s$ is conformally invariant in the following sense: for any $T\in \mG$ and for any $u\in \wmns$, we have $u_T\in \wmns$  and $I_s(u)=I_s(u_T)$,  where 
		$$
		u_T(x) :=u(Tx)+h_T(x).
		$$
	\end{theorem}

	The proof of Theorem \ref{Invariance-MinimumIs} strongly relies on a classification result obtained in  \cite{CirEspLi-2026} (see also \cite{Esp-2018} for the case $s=0$) for the following singular Liouville-type equation for $s \in [0,N)$:
	\begin{equation}\label{LiovilleProblem}
		\begin{cases}
			\dis -\Delta_N z=\frac{e^z}{|x|^s} \quad \text{ in $\mathbb{R}^N$,} \bigskip\\
			\dis \int_{\mathbb{R}^N}\frac{e^z}{|x|^s}\, dx<+\infty.
		\end{cases}
	\end{equation}
Specifically, any weak solution of \eqref{LiovilleProblem} can be written as $g_{N,s}\circ T+(N-s)\ln [T] +  \ln \left( \frac{\widetilde{\omega_{N,s}}}{ N}   \right)$  for some $T\in \mathcal G$. Here, we prove that critical points of $I_s$ are weak solutions to 
	\begin{equation*}
		-\Delta_N(g_{N,s}+u)=\frac{\widetilde{\omega_{N,s}}}{N}\frac{e^{u}\mu_{N,s}}{\int_{\mathbb{R}^N} e^u \, d\mu_{N,s}}= \frac{\widetilde{\omega_{N,s}}}{N}\frac{e^{u + g_{N,s} }/|x|^s}{\int_{\mathbb{R}^N} e^{u + g_{N,s}}/|x|^s \, dx}  \quad \text{ in $\mathbb{R}^N$}. 
	\end{equation*}
	Then, $u$ is a critical point of $I_s$ iff $w = u+g_{N,s}$ is a weak solution to  
	\begin{equation}\label{LiovilleWithIntegral}
		-\Delta_N w =	\frac{\widetilde{\omega_{N,s}}}{N} \frac{e^w/|x|^s}{\int_{\mathbb{R}^N} e^w/|x|^s \, dx} \text{ in $\mathbb{R}^N$},
	\end{equation}
	and the latter equation can be easily reduced to \eqref{LiovilleProblem}.

	In this paper we also investigate the optimality of the constant $1/\widetilde{\omega_{N,s}}$ in the inequality \eqref{MT-SingularityNewSpace}. For this purpose, for any $\lambda>0$, we consider the functional $I_{s,\lambda}: W_{\mu_{N,s}}(\mathbb{R}^N) \to \mathbb{R}$ defined as
	\begin{align}\label{functionalLambda}
		I_{s,\lambda}(u):=\frac{1}{\widetilde{\omega_{N,s}}}\int_{\mathbb{R}^N} H_{N,s}(u)\, dx + \lambda \int_{\mathbb{R}^N} u \, d\mu_{N,s}-\lambda\ln\left( \int_{\mathbb{R}^N} e^u \,  d\mu_{N,s} \right),
	\end{align}
	
	
	\noindent
	and we prove the following:

	\begin{theorem}\label{InfLambdaLess1Bigger1}
		Assume $N \geq 2$, $s \in [0,N)$ and $\lambda>0$. Then
		\begin{enumerate}[label=$\alph*$), ref=$\alph*$)]
			\item  if $0 < \lambda < 1$, we have
			\begin{equation}\label{InfZeroLambdaLeq1}
				\inf_{\wmns} I_{s,\lambda} =  0
			\end{equation} 
			
			and $I_{s,\lambda}(u) = 0 $ iff $u$ is a constant function; \label{lambda<1} \medskip
			\item  if $\lambda >1 $, we have $$\inf_{\wmns} I_{s,\lambda} = -\infty.$$\label{lambda>1}
		\end{enumerate}
	\end{theorem}

	The optimality of the constant $1/\widetilde{\omega_{N,s}}$ in  \eqref{MT-SingularityNewSpace} can also be deduced as a consequence of the following result, inspired by \cite{DolEstJan-2015}.

	\begin{proposition}\label{propN}
		Let $N \geq 2$, $s \in [0,N)$ and  $\lambda>0$. Then
		\[G_\lambda:=\inf_{\substack{u\in W_{\mu_{N,s}}(\mathbb{R}^N) \\ \int_{\mathbb{R}^N}u \, d\mu_{N,s}>0}} \left(\frac{\frac{1}{\widetilde{\omega_{N,s}}}\int_{\mathbb{R}^N}H_{N,s}(u)\, dx+\lambda \int_{\mathbb{R}^N}u \,  d\mu_{N,s} }{\ln \left(\int_{\mathbb{R}^N}e^u \, d\mu_{N,s}\right)} \right)=\min\{\lambda,1\}.\]
	\end{proposition}
	
	Finally, based on \cite{LamLu-2018}, we extend the following Poincaré inequality to $\wmns$, proving also that it is sharp.
	
	\begin{proposition}\label{Poincare}
		Let $u\in W_{\mu_{N,s}}(\mathbb{R}^N)$, and let us denote  $\overline{u}:=\int_{\mathbb{R}^N} u \, d\mu_{N,s}$. Then
		\begin{align*}
			\int_{\mathbb{R}^N}|u-\overline{u}|^2\, d\mu_{N,s}\leq& \frac{N(N-2)}{\widetilde{\omega_{N,s}}} \int_{\mathbb{R}^N}|\nabla g_{N,s}|^{N-4}(\nabla g_{N,s}\cdot \nabla u)^2\, dx\\&+\frac{N}{\widetilde{\omega_{N,s}}}\int_{\mathbb{R}^N}|\nabla g_{N,s}|^{N-2}|\nabla u|^2\, dx.
		\end{align*}
		Moreover, equality is achieved for the function $\phi_{N,s}\in \wmns$ defined by
		\begin{equation}\label{DefPhiNs}
			\phi_{N,s}(x):=\frac{N-s}{N-1}\left(\frac{N-1-|x|^{\frac{N-s}{N-1}}}{1+|x|^{\frac{N-s}{N-1}}}\right).
		\end{equation}
	\end{proposition}

	This paper is organized as follows: Section \ref{SectionSpace} is devoted to the properties of the space $\wmns$, and we prove Theorem \ref{densitystatementANDextension}. In addition, we obtain a Sobolev-type embedding result.
	Section \ref{SectionLiouville} focuses on the proof of Theorem \ref{Invariance-MinimumIs}, namely that the only critical points of the functional $I_s$ are minimizers, and the conformal invariance of $I_s$. To this end, we study the modified functional $I_{s,\lambda}$ defined in \eqref{functionalLambda} by proving  Theorem \ref{InfLambdaLess1Bigger1}, which allows us to obtain the optimality of the constant $1/\widetilde{\omega_{N,s}}$ in inequality \eqref{MT-SingularityNewSpace}. Such optimality is proved using different techniques in Section \ref{SectionOptimalConstant}, where we also extend the Poincaré inequality to $\wmns$. In Section \ref{SectionEquivalence}, we obtain the singular Carleson-Chang inequality \eqref{Chang-Carleson-Singularity}, and we prove in Theorem \ref{Equivalence} its equivalence with the singular $N$-dimensional Euclidean Onofri inequality \eqref{MT-SingularityNewSpace}. Finally, in Appendix \ref{app}, we detail the $C^1$ regularity of $I_{s,\lambda}$.
	
	\section{Framework space}\label{SectionSpace}
	
	\medskip
	\noindent
Given $N \geq 2$ and $s \in [0,N)$, 	let us define the weighted Sobolev space 
	\begin{align*}
		W_{\mu_{N,s}}(\mathbb{R}^N):=\{ \, u\in L^1(\mathbb{R}^N,d\mu_{N,s}): \; \;  |\nabla u|\in L^{N}(\mathbb{R}^N,dx); \; \;  |\nabla u|^2|\nabla g_{N,s}|^{N-2}\in L^1(\mathbb{R}^N,dx)\},
	\end{align*}
	endowed	with norm 
	\[\|u\|_{\mu_{N,s}}= \int_{\mathbb{R}^N} |u| \, d{\mu_{N,s}} + \|\nabla u\|_N + \left(\int_{\mathbb{R}^N}|\nabla u|^2 |\nabla g_{N,s}|^{N-2} \, dx\right)^{\frac{1}{2}}.\]
	Here and in the rest of the paper $\mu_{N,s}$,  $g_{N,s}$ and $H_{N,s}$ denote the quantities in \eqref{DefMuns} and \eqref{DefHns}. Moreover, $C$ refers to a generic positive constant whose value may change from line to line and its dependence on any specific variable will be indicated in the subscript (e.g. $C_{N}$, $C_{N,s}$).
	Let us recall the following result (see \cite[Lemma 2.1]{BorCinMan-2025}):
	
	\begin{lemma}\label{Lemma 2.1}
		For any $N\in \mathbb{N}$, $N \geq 2$, there exists a constant $C_{N}$ such that 
		$$ 0 \leq R_N(X,Y)\leq C_{N}\left(|Y|^N+|Y|^2|X|^{N-2}\right)  \qquad \forall \, X,Y\in \mathbb{R}^N . $$
		In particular, 
		$$
		0\le H_{N,s}(u)\le C_{N}(|\nabla u|^N + |\nabla u|^2 |\nabla g_{N,s}|^{N-2}).
		$$
	\end{lemma}
	
	The previous lemma directly implies the following integrability condition for $H_{N,s}$:
	
	\begin{corollary}
		Let $u\in  W_{\mu_{N,s}}(\mathbb{R}^N)$. Then 
		$$ \displaystyle \int_{\mathbb{R}^N}H_{N,s}(u)\,  dx<+\infty. $$
	\end{corollary}
	
	\begin{lemma}\label{Sobole-NewSpace}
		If $u\in W^{1,N}(\mathbb{R}^N)$ is such that
		\begin{equation*}
			|\nabla u|^2|\nabla g_{N,s}|^{N-2}\in L^1(\mathbb{R}^N, dx),
		\end{equation*} 
		then $u\in W_{\mu_{N,s}}(\mathbb{R}^N)$.
	\end{lemma}
	
	\begin{proof}
		
		Let $\displaystyle u \in  W^{1,N}(\mathbb{R}^N)$ such that $|\nabla u|^2|\nabla g_{N,s}|^{N-2}\in L^1(\mathbb{R}^N, dx)$.
		Then, $u\in W_{\mu_{N,s}}(\mathbb{R}^N)$ iff
		\begin{equation*}
			\int_{\mathbb{R}^N} \left| u\right| \, d\mu_{N,s} <+\infty.
		\end{equation*}
		Taking into account \eqref{DefMuns}, we have
		\begin{align*}
			\displaystyle
			& \mu_{N,s}(x) \sim |x|^{-s}, \quad \text{as } |x| \to 0;\\
			& \mu_{N,s}(x) \sim |x|^{-\frac{N^2-s}{N-1}}, \quad \text{as } |x| \to + \infty.
		\end{align*}
		
		\noindent
		Observe that $0 \leq s < N$ and $\frac{N^2-s}{N-1} > N$. Since $u \in  W^{1,N}(\mathbb{R}^N)$, by Sobolev embedding we know that $u \in L^t(\mathbb{R}^N,dx)$ for any $t \geq N$.
		Moreover, for any $s \in [0, N)$ there exists $q>1$ such that $sq < N$, $q\frac{N^2-s}{N-1} > N$ and $q'=\frac{q}{q-1} > N$. Therefore, Hölder's inequality with conjugate exponents $q$ and $q'$ implies that
		$$\displaystyle \int_{\mathbb{R}^N} |u| \, d\mu_{N,s}\leq \lVert u \rVert_{q'} \lVert \mu_{N,s} \rVert_{q} < + \infty.$$
		
	\end{proof}	
	
	\begin{remark}\label{proprietagradvnslnxs}
		Observe that for any $s \in [0,N)$, we have
		$$ \dis
		\nabla g_{N,s}(x) =-\frac{N(N-s)}{N-1} \frac{|x|^{\frac{1-s}{N-1}} }{1+|x|^{\frac{N-s}{N-1}}} \frac{x}{|x|}.
		$$
		In particular,
		\begin{align*}
			\left|\nabla g_{N,s}(x)\right|&\sim |x|^{-1}, \quad \text{as $|x|\rightarrow +\infty$;}
			\\
			\left|\nabla g_{N,s}(x)\right|&\sim |x|^{\frac{1-s}{N-1}}, \quad \text{as $|x|\rightarrow 0$.}
		\end{align*}
		Since $ \frac{N(1-s)}{N-1} > -N$, we get $| \nabla g_{N,s}|\in L^N_{loc}(\R^N,dx)$.
	\end{remark}
	
	\begin{remark}\label{Nlaplaciano}
		By direct computation it follows that
		\begin{equation*}
			\Delta_N g_{N,s}(x)=- \frac{\widetilde{\omega_{N,s}}}{N} \mu_{N,s}(x),
		\end{equation*}	
		for any $x \in \R^N\setminus\{0\}$.  Note that, if $N\ge 3$ or $N=2$ and $s>0$, the function $g_{N,s}$ is not of class $C^2$ near the origin. In fact, it is not of class $C^1$ if $s\ge 1$. Nevertheless, given a bounded open domain $\Omega\subseteq \R^N$ with smooth boundary, the standard Green identity 
		\begin{equation}\label{greenformula}
			\int_{\Omega}|\nabla g_{N,s}|^{N-2} \nabla g_{N,s} \cdot \nabla \varphi\, dx = \frac{\widetilde{\omega_{N,s}}}{N} \int_{\Omega} \varphi\, d\mu_{N,s} + \int_{\partial \Omega} \varphi\, |\nabla g_{N,s}|^{N-2} \nabla g_{N,s} \cdot\nu \, d\sigma\,
		\end{equation}
		holds even if $0\in \Omega$, for all functions $\varphi \in C^1(\overline{\Omega}\setminus\{0\})$ with $\varphi \in L^1(\Omega,d\mu_{N,s})$, $|\nabla \varphi ||\nabla g_{N,s}|^{N-1}\in L^1(\Omega,dx)$ and 
		$$
		|\varphi(x)| = o( |x|^{s-N} )  \quad \text{ as } |x|\to 0.
		$$
		This can be easily proved by applying the Green identity on $\Omega \setminus B_{\varepsilon}(0)$, for $\varepsilon$ sufficiently small, and  then passing to the limit as $\varepsilon \to 0$. Note that, by Remark \ref{proprietagradvnslnxs}, \eqref{greenformula} can be applied for example when $\varphi = g_{N,s}$. Moreover, by approximation with functions in $C^1(\overline \Omega)$ and the trace theorem, \eqref{greenformula} also holds for functions  $\varphi	\in W^{1,N}(\Omega)$.
	\end{remark}

	We now give the proof of Theorem \ref{densitystatementANDextension}. The starting point is the density result 
	\[W_{\mu_{N,s}}(\mathbb{R}^N)=\overline{C^{\infty}_0(\mathbb{R}^N)}^{\|\cdot\|_{\mu_{N,s}}}.\]
	To prove this result, we follow \cite{BorCinMan-2025} by first showing that any bounded function in $\wmns$ belongs to $\overline{C^{\infty}_0(\mathbb{R}^N)}^{\Vert{}\cdot\Vert{}_{\mu_{N,s}}}$ and then approximating via truncated functions. Since the argument is similar to that in \cite{BorCinMan-2025}, we only provide a sketch, highlighting the main novelties.
	
	\begin{proof}[\textbf{Proof of Theorem \ref{densitystatementANDextension}} ]
		As a first step, we show that for every $u_{\infty}\in W_{\mu_{N,s}}(\mathbb{R}^N)\cap L^{\infty}(\mathbb{R}^N)$, there exists a sequence   $ \{\tilde u_k\}_k \subset C^{\infty}_0(\mathbb{R}^N)$, such that  $\|\tilde u_k- u_{\infty}\|_{\mu_{N,s}}\rightarrow 0$. With this aim, for $k\in \mathbb N, k\ge 2$, we consider
		\begin{equation*}
			\eta_k(x):=
			\begin{cases}
				\displaystyle \frac{1}{\sqrt[k]{k}} & \text{ if } |x| \leq 1,  \\
				\displaystyle  \frac{1}{\sqrt[k]{|x|}} + \frac{1}{\sqrt[k]{k}} - 1 & \text{ if }  1 < |x| \leq k^*,  \\
				0 & \text{ if }  k^* < |x| ,
			\end{cases}
		\end{equation*}
		where $k^*:=1/\left(1-\frac{1}{\sqrt[k]{k}}\right)^k$ and  $u_k := u_\infty \eta_k $. Since $ u_{\infty}\in  L^\infty(\mathbb{R}^N)$, $ 0  \leq \eta_k \leq 1$, and $ \eta_k(x) \to 1 $ as $k \to +\infty $ for any $x \in \mathbb{R}^N$, from the Dominated Convergence Theorem we have $u_k \to u_{\infty}$ in $L^1(\mathbb{R}^N,d\mu_{N,s})$. Moreover,
		\begin{align}\label{Conv_B1}
			\displaystyle 
			& \int_{B_1} | \nabla(u_k-u_{\infty})|^2|\nabla g_{N,s}|^{N-2} \, dx = \left( 1 - \frac{1}{\sqrt[k]{k}} \right)^2 \int_{B_1} | \nabla u_{\infty}|^2|\nabla g_{N,s}|^{N-2} \, dx \rightarrow 0.
		\end{align}
		\noindent
		In addition, from \cite[Lemma 2.6]{BorCinMan-2025}, $\|\nabla ( u_{\infty} - u_k) \|_N \to 0$ and
		\begin{align}\label{Conv_Out_B1}
			\int_{\mathbb{R}^N}|\nabla(u_k-u_{\infty})|^2|\nabla g_{N,0}|^{N-2} \, dx\rightarrow 0.
		\end{align}
		\noindent
		Note that, since $g_{N,s}$ and $g_{N,0}$ have the same asymptotic behavior at infinity,  see Remark \ref{proprietagradvnslnxs}, we have in particular
		$|\nabla g_{N,s}|\leq C|\nabla g_{N,0}|$ in $\mathbb{R}^N\setminus B_1$. Then,  from \eqref{Conv_B1} and \eqref{Conv_Out_B1}, we get that  $|\nabla(u_k-u_\infty)|^2|\nabla g_{N,s}|^{N-2}\rightarrow 0$ in $L^1(\mathbb{R}^N,dx)$. Therefore $\|u_k-u_{\infty}\|_{\mu_{N,s}}\rightarrow 0$.\\
		We now approximate the sequence  $\{u_k\}_k$ with a sequence of smooth compactly supported functions. Observe that ${u_{k}}_{|_{B_{k^*}}} \in W_0^{1,N}(B_{k^*}),$ so there exists $\tilde{u}_k \in C_0^{\infty}(B_{k^*})$ such that 
		$$ \displaystyle \left\lVert {u_{k}}_{|_{B_{k^*}}}- \tilde{u}_k  \right\rVert_{L^N(B_{k^*},dx)} + \left\lVert \nabla \left({u_{k}}_{|_{B_{k^*}}}- \tilde{u}_k \right) \right\rVert_{L^N(B_{k^*},dx)} < \frac{1}{k}.$$
		We can extend to $0$ in $\mathbb{R}^N \setminus B_{k^*}$  and consider $\tilde{u}_k \in C_0^{\infty}(\mathbb{R}^N)$. Hence
		\begin{equation}\label{dis2} 
			\displaystyle \lVert u_k - \tilde{u}_k  \rVert_N + \lVert \nabla \left(u_k - \tilde{u}_k \right) \rVert_N  < \frac{1}{k}.
		\end{equation}
		
		\noindent
		As in the proof of Lemma \ref{Sobole-NewSpace}, let $q>1$ be such that $\mu_{N,s}\in L^q(\R^N,dx)$ and $q' = \frac{q}{q-1}\ge N$. Due to H\"older inequality and Sobolev embedding, we have
		\begin{align*}
			\displaystyle 
			\int_{\mathbb{R}^N} |u_k - \tilde{u}_k| \mu_{N,s} \, dx 
			& \leq \lVert u_k - \tilde{u}_k\rVert_{{q'}} \lVert \mu_{N,s}  \rVert_{{q}} \\
			& \leq C_{N,q}\left(\lVert  u_k -  \tilde{u}_k\rVert_{N}+ \lVert \nabla u_k - \nabla \tilde{u}_k\rVert_{N}\right)\lVert \mu_{N,s}  \rVert_{{q}} \\
			& < \frac{C_{N,q} \lVert \mu_{N,s}  \rVert_{q}}{k}=o(1),
		\end{align*} 
		
		\noindent
		where $o(1)$ denotes any quantity that goes to $0$ as $k \to + \infty$. If $N=2$, we're done. Otherwise, if $N \geq 3$, we have
		\begin{align*}
			\int_{0<|x|\leq k^*} |\nabla g_{N,s}(x)|^{N} \, dx 
			& = C_{N,s}\int_0^{k^*} \frac{\rho^{\frac{(1-s)N}{N-1}} }{(1+\rho^{\frac{N-s}{N-1}})^N}  \rho^{N-1} \, d \rho \\
			& \leq C_{N,s}\int_0^{1}  \frac{1}{\rho^{ \frac{N(s-1)}{N-1} +1 - N  }}   \, d \rho 
			+  C_{N,s}\int_1^{k^*} \frac{1}{\rho}   \, d \rho.
		\end{align*}
		
		\noindent
		Furthermore, since $s \in [0,N)$, we have  $N(s-1)/(N-1)+1 - N < 1$. Therefore,
		$$ \displaystyle \int_{0<|x|\leq k^*} |\nabla g_{N,s}(x)|^{N} \, dx \leq C_{N,s}( \ln \, k^* + 1).$$
		
		\noindent
		Then, due to \eqref{dis2} and H\"older inequality with exponents $\frac{N}{2}$ and $\frac{N}{N-2}$,
		\begin{align*}
			\displaystyle
			& \int_{\mathbb{R}^{N}}|\nabla (u_k-\tilde{u}_{k})|^2|\nabla g_{N,s}|^{N-2} \, dx 
			\leq  C_{N,s}\frac{(\ln k^*+1)^{\frac{N-2}{N}}}{k^2}.
		\end{align*}
		
		\noindent
		Moreover,
		\begin{align*}
			\displaystyle
			& \lim_{k \to + \infty} \frac{(\ln k^*+1)^{\frac{N-2}{N}}}{k^2}=0.
		\end{align*}
		
		\noindent
		Hence
		$$ \displaystyle \int_{\mathbb{R}^N}  |\nabla(u_k - \tilde{u}_k)|^2 |\nabla g_{N,s}|^{N-2} \, dx \to 0.$$
		
		\noindent
		Thus we have a sequence $\{\tilde{u}_k \}_k \subset C_0^{\infty}(\mathbb{R}^N)$ such that $\lVert u_{\infty} -\tilde{u}_k \rVert_{\mu_{N,s}} \to 0$ as $k \to +\infty$ and
		\begin{align*}
			\displaystyle
			\lVert u_{\infty} - \tilde{u}_k \rVert_{\mu_{N,s}} \leq  \lVert u_{\infty} -  u_k \rVert_{\mu_{N,s}} +  \lVert u_k - \tilde{u}_k \rVert_{\mu_{N,s}} \to 0 \quad \text{as } k \to + \infty.  
		\end{align*}
		
		Now, from \cite{BorCinMan-2025}, taking into account that  $dx$ is absolutely continuous with respect to the measure $d\mu_{N,s}$, we deduce that for every $u\in W_{\mu_{N,s}}(\mathbb{R}^N)$, there exists a sequence $u_{\infty,k}\in W_{\mu_{N,s}}(\mathbb{R}^N)\cap L^{\infty}(\mathbb{R}^N)$ such that $\|u-u_{\infty,k}\|_{\mu_{N,s}}\rightarrow 0$, when $k\rightarrow +\infty$. As a consequence, we obtain
		\begin{equation*}
			W_{\mu_{N,s}}=\overline{C^{\infty}_0(\mathbb{R}^N)}^{\|\cdot\|_{\mu_{N,s}}}.
		\end{equation*}
		Finally, inequality \eqref{MT-SingularityNewSpace} follows directly from applying the Dominated Convergence Theorem. Indeed, if $u\in \wmns$ and $u_n\in C^\infty_{0}(\R^N)$ satisfy $u_n\to u$ in $\wmns$, then, up to a subsequence, we have:
		$\nabla u_n \to \nabla u$ a.e. in $\R^N$ and 
		$$
		|\nabla u_n|^N + |\nabla u_n|^2 |\nabla g_{N,s}|^{N-2} \le U \in L^1(\R^N,dx). 
		$$
		Then by Lemma \ref{Lemma 2.1}, we get $H_{N,s}(u_n)\to H_{N,s}(u)$ in $L^1(\R^N,dx)$. Using the convergence of $u_n$ to $u$ in $L^1(\R^N,d\mu_{N,s})$ and  Fatou's lemma to estimate the exponential term, we get \eqref{MT-SingularityNewSpace}.
	\end{proof}
	
	Next, we use  the previous extension to deduce the continuous embedding of $\wmns$ in the space $L^q(\mathbb{R}^N, d \mu_{N,s})$ for every $1 \leq q < \infty$.

	\begin{proposition}\label{inclusionLq}
		Assume $N \geq 2$, $s\in [0,N)$ and $ 1 < q < \infty$. Then 
		$$ \dis W_{\mu_{N,s}}(\mathbb{R}^N) \hookrightarrow L^q( \mathbb{R}^N, d \mu_{N,s}).$$
	\end{proposition}
	
	\begin{proof}
		Let $u\in \wmns$ be such that $\|u\|_{\mu_{N,s}}\le 1$. Taking into account \cite[Lemma A.4, Chapter 2]{KS}, we have  $ |u| \in W_{\mu_{N,s}}(\mathbb{R}^N)$, so using Theorem \ref{densitystatementANDextension},
		\begin{equation*}
			\ln\left( \int_{\mathbb{R}^N} e^{|u|} \, d\mu_{N,s} \right)\leq \frac{1}{\widetilde{\omega_{N,s}}}\int_{\mathbb{R}^N} H_{N,s}(|u|) \, dx + \int_{\mathbb{R}^N} |u| \, d\mu_{N,s}.
		\end{equation*}
		
		\noindent
		By applying Lemma \ref{Lemma 2.1} and once again \cite[Lemma A.4, Chapter 2]{KS}, there exists $C_{N} >0$ such that for almost every $x \in \mathbb{R}^N$ we get
		$$ \dis  H_{N,s}(|u(x)|)  \leq C_{N} (|\nabla u(x)|^N + | \nabla u(x)|^2 | \nabla g_{N,s}(x)|^{N-2}).$$
		As a consequence, since  $\|u\|_{\mu_{N,s}}\le 1$, for some $a_{N,s}>0$, we have
		\begin{align*}
			\ln\left( \int_{\mathbb{R}^N} e^{|u|} \, d\mu_{N,s} \right)\leq   a_{N,s} \|u\|_{\mu_{N,s}},
		\end{align*}
		
		\noindent
		by which
		\begin{align*}
			\int_{\mathbb{R}^N} e^{|u|} \, d\mu_{N,s} \leq e^{a_{N,s}\|u\|_{\mu_{N,s}}}.
		\end{align*}
		
		\noindent
		Furthermore, recalling the Taylor expansion of the exponential, we infer that for any integer $q\geq 1$
		$$\dis 1+ \frac{1}{q!} |u|^q \leq  \sum_{n=0}^{\infty} \frac{|u|^n}{n!} = e^{|u|}.$$
		
		\noindent
		Altogether, we have proved
		$$\dis \intrn |u|^q  \, d\mu_{N,s}  \leq  q!\left( e^{a_{N,s}\|u\|_{\mu_{N,s}}}-1\right) \le q! \left( e^{a_{N,s}}-1\right), $$
		which implies that $ \dis W_{\mu_{N,s}}(\mathbb{R}^N) \hookrightarrow L^q( \mathbb{R}^N, d \mu_{N,s}).$ By interpolation, we get the thesis for any real number $q \geq 1$. 
		
	\end{proof}

	\section{Liouville type problems and properties of $I_{s,\lambda}$}\label{SectionLiouville}

	In this Section, we focus on the connection between the critical points of the functional $I_{s}$ and the solutions of the Liouville-type equation \eqref{LiovilleWithIntegral}. In particular, we are specially interested in the case $\int_{\mathbb{R}^N}e^w/|x|^s \, dx=1$, so that \eqref{LiovilleWithIntegral} rewrites as
	\begin{equation}\label{LiovilleNoIntegral}
		-\Delta_N w= \frac{\widetilde{\omega_{N,s}}}{N} \frac{e^w}{|x|^s} \text{ in $\mathbb{R}^N$}.
	\end{equation}
	
	In the literature it is possible to find different concepts of weak solution to \eqref{LiovilleWithIntegral}. Throughout the paper, we deal with the following notion.
	\begin{definition}\label{solution}
		A function $w\in W^{1,N}_{loc}(\mathbb{R}^N)$ such that $e^w/|x|^s\in L^1(\mathbb{R}^N,dx)$, is a weak solution of \eqref{LiovilleWithIntegral} if it satisfies 
		\begin{equation*}
			\int_{\mathbb{R}^N}|\nabla w|^{N-2}\nabla w\cdot\nabla \rho \, dx=  \frac{\widetilde{\omega_{N,s}}}{N} \frac{\int_{\mathbb{R}^N}e^w\rho/|x|^s \, dx}{\int_{\mathbb{R}^N}e^w/|x|^s \, dx}
		\end{equation*}
		for all $\rho\in C^{\infty}_0(\mathbb{R}^N)$.
	\end{definition}
	\begin{remark}
		It is equivalent if the test functions belong to $C^1_{0}(\mathbb{R}^N)$, so the definition of solution coincides with the notion of solutions in \cite{Esp-2021} and  \cite{PraTar-2001}. Moreover, by density, it is equivalent if $$\rho \in \overline{H}:=\{\rho\in W^{1,N}_0(\Omega):\Omega\subset \mathbb{R}^N \text{ bounded}\},$$ so it also coincides with the concept dealt in \cite{Esp-2018} for the nonsingular case. 
	\end{remark}

	We point out that the functional $I_{s,\lambda}$ is of class $C^1$ on the space $\wmns$. In particular, a function $u\in \wmns$ is a critical point of $I_{s,\lambda}$ iff: 
	\begin{equation}\label{criticalpoints}
		\begin{aligned}
			\frac{N}{\widetilde{\omega_{N,s}}}\int_{\mathbb{R}^N}  \Bigl[ \, |\nabla g_{N,s}+\nabla u|^{N-2}(\nabla g_{N,s}+\nabla u)  - |\nabla g_{N,s}|^{N-2}\nabla g_{N,s} \, \Bigr]  \cdot \nabla v \, dx \\
			= \lambda\frac{\dis \int_{\mathbb{R}^N}e^u v \, d\mu_{N,s}}{ \dis \int_{\mathbb{R}^N}e^u \, d\mu_{N,s}}- \lambda\int_{\mathbb{R}^N} v \, d\mu_{N,s},
		\end{aligned}
	\end{equation}
	for any $v\in \wmns$. For the reader's convenience, the proof of the $C^1$ regularity is postponed to Appendix \ref{app}.

	\begin{remark}\label{app.Critical}
		If $u\in W_{\mu_{N,s}}(\mathbb{R}^N)$, then
		$$ \langle I_{s,\lambda}'(u), v \rangle =0  \quad \forall \,  v \in \wmns \qquad \iff  \qquad   \langle I_{s,\lambda}'(u), v \rangle =0  \quad \forall \, v \in C^{\infty}_0(\mathbb{R}^N).$$	
		The implication “ $\Rightarrow"$ is immediate. Conversely, we note that for any $v \in \wmns$ there exists a sequence $\{v_n\}_n \subset C_0^{\infty}(\mathbb{R}^N)$ such that $\lVert v_n - v \rVert_{\mu_{N,s}} \to 0$, and then we apply the continuity of the map $I_{s,\lambda}'(u): \wmns \to \mathbb{R}$.	
	\end{remark}

	\begin{lemma}\label{ConectionLioville}
		Let $N\geq 2$, and $s\in[0,N)$. A function $u\in W_{\mu_{N,s}}(\mathbb{R}^N)$ is a critical point of $I_{s}$ if and only if $w=u+g_{N,s}$ is a  weak solution to \eqref{LiovilleWithIntegral}. 
	\end{lemma}
	\begin{proof}		
		First, by $g_{N,s} \in L_{loc}^{\infty}(\mathbb{R}^N)$ and $|\nabla g_{N,s}| \in L^N_{loc}(\mathbb{R}^N,dx)$, we get $g_{N,s}\in W^{1,N}_{loc}(\mathbb{R}^N)$.  Notice now that any  $u\in W_{\mu_{N,s}}(\mathbb{R}^N)$ is such that $\left| \nabla u \right| \in L_{loc}^N (\mathbb{R}^N, dx)$. 
		In addition,  for any ball $B\subset \mathbb{R}^N$ we have $\mu_{N,s}^{-1}\in L^{\infty }(B)$, and by Proposition \ref{inclusionLq} we deduce
		\begin{align*}
			\int_{B} |u|^N \, dx =&\int_{B} \mu_{N,s}^{-1} |u|^N\mu_{N,s} \, dx\leq  \| \mu_{N,s}^{-1}\|_{L^{\infty}(B)}\int_{\mathbb{R}^N} |u|^N\, d\mu_{N,s}<+\infty.
		\end{align*}
		As a consequence, $u+g_{N,s}\in W^{1,N}_{loc}(\mathbb{R}^N)$. Moreover, since $u$ is a critical point of $I_s$,  for any $\rho\in C^{\infty}_0(\mathbb{R}^N)$ we have		
		\begin{align*}
			\langle  I'_{s}(u), \rho \rangle=&\frac{N}{\widetilde{\omega_{N,s}}}\int_{\mathbb{R}^N}\bigg(|\nabla g_{N,s}+\nabla u|^{N-2}(\nabla  g_{N,s}+\nabla u) -|\nabla g_{N,s}|^{N-2}\nabla g_{N,s}\bigg)\cdot\nabla \rho \, dx
			\\&+\int_{\mathbb{R}^N}\rho \,d\mu_{N,s} -\frac{\int_{\mathbb{R}^N}e^u\rho\, d\mu_{N,s}}{\int_{\mathbb{R}^N}e^u\,d\mu_{N,s}}=0,
		\end{align*}
		that is $u$ is a weak solution of
		\begin{equation*}
			-\frac{N}{\widetilde{\omega_{N,s}}}\left(\Delta_N(g_{N,s}+u)-\Delta_N g_{N,s} \right)+\mu_{N,s}-\frac{e^u\mu_{N,s}}{\int_{\mathbb{R}^N} e^u \, d\mu_{N,s}}=0.
		\end{equation*}
		By Remark \ref{Nlaplaciano}, we deduce that $u$ solves in a weak sense
		\begin{equation*}
			-\Delta_N(g_{N,s}+u)=\frac{\widetilde{\omega_{N,s}}}{N}\frac{e^u\mu_{N,s}}{\int_{\mathbb{R}^N} e^u \, d\mu_{N,s}}=\frac{\widetilde{\omega_{N,s}}}{N}\frac{e^{u+g_{N,s}}/|x|^s}{\int_{\mathbb{R}^N} e^{u+g_{N,s}}/|x|^s \, dx}.
		\end{equation*}
		Therefore, $w=u+g_{N,s}$ is a weak solution of \eqref{LiovilleWithIntegral}.

		Conversely, if $w = u +g_{N,s}$ is a weak solution of \eqref{LiovilleWithIntegral} with $u\in \wmns$,  by using Remark \ref{app.Critical} we get that  $u$ is a critical point of $I_s$. 
	\end{proof}

	Before studying the behavior of $I_{s,\lambda}$ under the transformations in  Theorem \ref{Invariance-MinimumIs}, let us state some technical preliminary  results. Throughout this section, we let $\mG$ and $h_T$ be defined as in \eqref{G} and \eqref{notation}. 
	
	\begin{lemma}\label{prophT}
		For any $T \in \mG$, we have $h_T \in \wmns$. Moreover, we have
		\begin{itemize}
			\item $\displaystyle{\lim_{|x|\to +\infty} h_T(x) = -\frac{N-s}{N-1}\ln[T]}$;
			\item If $s=0$ and $T=T_{t,x_0}$ with $x_0 \neq 0$, then  $|\nabla h_T | = O\left(\frac{|\nabla g_{N,s}|}{|x|}\right)$ as $|x|\to +\infty$;
			\item If $T=T_{t,0}$ with $t\neq 1$, then
			\begin{equation}\label{gradientHTinfinity}
				|\nabla h_T(x)| \sim |x|^{\frac{1+s-2N}{N-1}}\quad \text{ as $|x|\rightarrow +\infty,$ }
			\end{equation}
			and 
			\begin{equation}\label{gradientHTzero}
				|\nabla h_T(x)|  \sim |x|^{\frac{1-s}{N-1}}\quad \text{ as $|x|\rightarrow 0.$ }
			\end{equation}
		\end{itemize} 
		Moreover, the  identity 
		\begin{equation}\label{fundamentalid}
			\mu_{N,s}(x) e^{h_T(x)} = [T]^N \mu_{N,s}(T x)
		\end{equation} 
		holds for any $x\in \R^N$ if $s=0$, and for any $x\in \R^N\setminus \{0\}$ if $s>0$.
	\end{lemma}
	\begin{proof}
		If $T\in \mG$, then there exist $t>0$ and $x_0\in \R^N$ such that $h_T(x) = g_{N,s}(t (x-x_0))- g_{N,s}(x)+(N-s) \ln t$, with $x_0 =0$ when $s>0$.  Using the explicit expression of $g_{N,s}$ we get
		\begin{align*}
			h_T(x) =N\ln\left(\frac{1+|x|^{\frac{N-s}{N-1}}}{1+|t(x-x_0)|^{\frac{N-s}{N-1}}}\right) + (N-s) \ln t 
		\end{align*}
		from which we get
		$$
		\lim_{|x|\to +\infty} h_T(x) = \left(-N\frac{N-s}{N-1} +N-s\right) \ln t = - \frac{N-s}{N-1}\ln t.
		$$
		Since $h_T \in C(\R^N)$, we conclude that $h_T\in L^\infty(\R^N)\subseteq L^1(\R^N,d\mu_{N,s})$.
		Still by direct computation, we have 
		$$\begin{aligned}
			\nabla h_T(x)
			=
			\frac{N(N-s)}{N-1}
			\left[
			\frac{
				|x|^{\frac{1-s}{N-1}}
			}{
				1+|x|^{\frac{N-s}{N-1}}
			} \frac{x}{|x|}
			-
			\frac{
				t^{\frac{N-s}{N-1}}
				|x-x_0|^{\frac{1-s}{N-1}}
			}{
				1+t^{\frac{N-s}{N-1}}
				|x-x_0|^{\frac{N-s}{N-1}}
			}\frac{x-x_0}{|x-x_0|}
			\right].
		\end{aligned}
		$$
		If $x_0 = 0$ (which is true whenever $s>0$), we directly get
		\begin{equation*}
			\nabla h_T(x) =  	\left(\frac{t^{\frac{N-s}{N-1
				}} - 1}{1+|tx|^{\frac{N-s}{N-1}}}\right) \nabla g_{N,s}(x).
		\end{equation*}
		Hence \eqref{gradientHTinfinity},  and \eqref{gradientHTzero} hold by Remark \ref{proprietagradvnslnxs}.
		
		If $s=0$ and $x_0 \neq 0$, we have 
		$$
		|\nabla h_T| = |\nabla g_{N,s}| \left| \frac{x}{|x|}- \frac{
			t^{\frac{N}{N-1}}
			|x-x_0|^{\frac{1}{N-1}}
			\left(1+|x|^{\frac{N}{N-1}}\right)}{
			\left(	1+t^{\frac{N}{N-1}}
			|x-x_0|^{\frac{N}{N-1}}\right)|x|^{\frac{1}{N-1}}
		}\frac{x-x_0}{|x-x_0|}\right| =O\left(\frac{|\nabla g_{N,s}|}{|x|}\right).$$
		In any case, it follows that $|\nabla h_T|\in L^N(\R^N, dx)$ and $|\nabla h_T|^{2}|\nabla g_{N,s}|^{N-2}\in L^1(\R^N, dx)$, so that $h_T\in \wmns$.  Finally, we observe that 
		$$
		[T]^N \mu_{N,s}(Tx) = [T]^N \frac{e^{g_{N,s}(Tx)}}{|Tx|^s} = \frac{e^{h_T(x)+g_{N,s}(x)}[T]^s}{|Tx|^s}= e^{h_T(x)}\mu_{N,s}(x) \frac{|x|^s[T]^s}{|Tx|^s}.
		$$
		Since either $s = 0$, or $x_0 =0$ and  $Tx= [T] x$, the fundamental identity \eqref{fundamentalid} is proved.

	\end{proof}

	From now on, we set
	\begin{equation*}
		\J=\frac{N}{\widetilde{\omega_{N,s}}}\int_{\mathbb{R}^N}|\nabla g_{N,s}|^{N-2}\nabla g_{N,s}\cdot \nabla h_T \, dx.
	\end{equation*}

	\begin{proposition}\label{functionalwithu=0} 
		Assume $s\in [0,N)$. For any $T\in \mG$,  we have		
		\begin{align}\label{integralht}
			\int_{\mathbb{R}^N} h_T \, d\mu_{N,s}=& \J- \frac{N-s}{N-1} \ln [T],
		\end{align}
		and
		\begin{align*}
			\frac{1}{\widetilde{\omega_{N,s}}}\int_{\mathbb{R}^N} H_{N,s}(h_T)\, dx=&	\frac{N-s}{N-1}\ln [T]-\J.
		\end{align*}
	\end{proposition}
	
	\begin{proof}
		By Remark \ref{Nlaplaciano} and Lemma \ref{prophT}, for any $r>0$, we can use the Green's identity to get
		\begin{align*}
			\int_{\mathbb{R}^N} h_T\, d\mu_{N,s}=&\lim_{r\rightarrow +\infty}	\int_{B_r} h_T\, d\mu_{N,s} \\=& \mathcal{J}_{N,s,T}-
			\frac{N}{\widetilde{\omega_{N,s}}}\lim_{r\rightarrow +\infty}\left(\int_{\partial B_r}h_T(x)|\nabla g_{N,s}(x)|^{N-2}\nabla g_{N,s}(x)\cdot \frac{x}{|x|}\, d\sigma\right). 
		\end{align*}
		Still from Lemma \ref{prophT}, we have
		\begin{align*}
			-\frac{N}{\widetilde{\omega_{N,s}}}&\lim_{r\rightarrow +\infty}\left(\int_{\partial B_r}h_T(x)|\nabla g_{N,s}(x)|^{N-2}\nabla g_{N,s}(x)\cdot \frac{x}{|x|}\, d\sigma\right)\\
			&=	\lim_{r\rightarrow +\infty} \left(-\frac{N-s}{N-1}\ln [T]+o(1)\right)\left(\frac{N}{\widetilde{\omega_{N,s}}}\int_{\partial B_r}|\nabla g_{N,s}(x)|^{N-1} \, d\sigma\right)\\
			&=	\lim_{r\rightarrow +\infty} \left(-\frac{N-s}{N-1}\ln [T]+o(1)\right)\left( 1 + o(1) \right)\\
			&= - \frac{N-s}{N-1} \ln [T].
		\end{align*}
		Thus, \eqref{integralht} is proved. On the other hand, by the definitions of $h_T$, $H_{N,s}$ and $\mathcal{J}_{N,s,T}$, we get
		$$
		\begin{aligned} 
			\frac{1}{\widetilde{\omega_{N,s}}}&\int_{\mathbb{R}^N} H_{N,s}(h_T)\, dx\\=&	\frac{1}{\widetilde{\omega_{N,s}}}\lim_{r\rightarrow +\infty }\left(  \int_{B_r}|\nabla ( g_{N,s}\circ T)|^N \, dx - \int_{B_r}|\nabla g_{N,s}|^Ndx \right)-\mathcal{J}_{N,s,T}\, ,
		\end{aligned}
		$$
		and, using again Remark \ref{Nlaplaciano}, we have
		$$
		\begin{aligned}
			\frac{N}{\widetilde{\omega_{N,s}}}	\int_{B_r}|\nabla g_{N,s}|^N  dx &  =  \int_{B_r} g_{N,s} \mu_{N,s} dx  +  \frac{N}{\widetilde{\omega_{N,s}}} \int_{\partial B_r} g_{N,s} |\nabla g_{N,s}|^{N-2}\nabla g_{N,s}\cdot \frac{x}{|x|} \, d \sigma \\
			& = \int_{\R^N} g_{N,s} \mu_{N,s} dx +o(1) + 
			\frac{r^{N-s}}
			{\left(1+r^{\frac{N-s}{N-1}}\right)^{N-1}}
			\ln\left[
			\frac{\omega_{N-1}}{N-s}
			\left(1+r^{\frac{N-s}{N-1}}\right)^N\right]\\
			& = \int_{\R^N} g_{N,s} \mu_{N,s} dx +\frac{N(N-s)}{N-1}\ln r + \ln \left(\frac{\omega_{N-1}}{N-s}\right) + o(1)\,
		\end{aligned}
		$$
		as $r\to +\infty$. 
		Similarly, denoting by $g_{N,s}^T := g_{N,s}\circ T$ and $\mu_{N,s}^T := \mu_{N,s}\circ T$, we find
		$$
		\begin{aligned}
			\frac{N}{\widetilde{\omega_{N,s}}}	\int_{B_r}|\nabla g_{N,s}^T|^N  dx &  =  \int_{B_r} [T]^N g_{N,s}^T \mu_{N,s}^T \,  dx  +  \frac{N}{\widetilde{\omega_{N,s}}} \int_{\partial B_r} g_{N,s}^T |\nabla g_{N,s}^T|^{N-2}\nabla g_{N,s}^T \cdot \frac{x}{|x|} \, d \sigma\\
			&  =  \int_{\R^N}  g_{N,s} \mu_{N,s} \,  dx  +  o(1) + \frac{N}{\widetilde{\omega_{N,s}}} \int_{\partial B_r} g_{N,s}^T |\nabla g_{N,s}^T|^{N-2}\nabla g_{N,s}^T \cdot \frac{x}{|x|} \, d \sigma \,.
		\end{aligned}
		$$
		
		\medskip
		\noindent
		But on $\partial B_r$ we have the  following uniform estimate as $r\to +\infty$:
		$$ \begin{aligned} & g_{N,s}^T(x) \left|\nabla g_{N,s}^T(x)\right|^{N-2} \nabla g_{N,s}^T(x)\cdot \frac{x}{|x|} \\  & = \left(\frac{N(N-s)}{N-1}\right)^{N-1} \frac{t^{N-s} |x-x_0|^{1-s}} {\left( 1+t^{\frac{N-s}{N-1}} |x-x_0|^{\frac{N-s}{N-1}} \right)^{N-1}} \frac{x\cdot(x-x_0)} {r\,|x-x_0|}  \ln\left[ \frac{\omega_{N-1}}{N-s} \left( 1+[T]^{\frac{N-s}{N-1}} |x-x_0|^{\frac{N-s}{N-1}} \right)^N \right] \\
			& =\left(\frac{N(N-s)}{N-1}\right)^{N-1} \frac{1}{r^{N-1}} \left[ \frac{N(N-s)}{N-1}\ln r + \frac{N(N-s)}{N-1}\ln [T] + \ln\left(\frac{\omega_{N-1}}{N-s}\right) +o(1) \right]\end{aligned} $$
		as $r\to +\infty$. Therefore we get 
		$$
		\begin{aligned}
			\frac{N}{\widetilde{\omega_{N,s}}}	\int_{B_r}|\nabla g_{N,s}^T|^N  &  =  \int_{\R^N} g_{N,s} \mu_{N,s} dx +\frac{N(N-s)}{N-1}\ln r  +\frac{N(N-s)}{N-1}\ln [T] + \ln \left(\frac{\omega_{N-1}}{N-s}\right) + o(1).
		\end{aligned}
		$$
		We then conclude that 
		$$ 
		\frac{1}{\widetilde{\omega_{N,s}}}\int_{\mathbb{R}^N} H_{N,s}(h_T)\, dx= \frac{N-s}{N-1}\ln [T]-\mathcal{J}_{N,s,T}.
		$$

	\end{proof}
	\begin{corollary}\label{CorImportant}
		For any $s \in [0,N)$, $\lambda>0$ and $T\in \mG$, we have 
		$$
		I_{s,\lambda}(h_T)=(\lambda-1)\left(\J - \frac{N-s}{N-1}\ln [T]\right).
		$$
		In particular, taking $\lambda=1$, we get $I_{s}(h_T) = 0$.
	\end{corollary}
	\begin{proof}
		By definition 
		$$
		I_{s,\lambda}(h_T)=\frac{1}{\widetilde{\omega_{N,s}}}\int_{\mathbb{R}^N} H_{N,s}(h_T)\, dx + \lambda \int_{\mathbb{R}^N} h_T \, d\mu_{N,s}-\lambda\ln\left( \int_{\mathbb{R}^N} e^{h_T} \,  d\mu_{N,s} \right).$$
		By means of \eqref{fundamentalid} and making the change of variable $Tx=y$, we have
		\begin{align*}
			\ln\left(\int_{\mathbb{R}^N}e^{h_T}\, d\mu_{N,s}\, \right)=\ln\left(\int_{\mathbb{R}^N}\mu_{N,s}(Tx)[T]^N \, dx\right)=\ln\left(\int_{\R^N} \mu_{N,s} dy \right) = 0.
		\end{align*}
		Thus, the conclusion easily follows from the previous proposition. 
	\end{proof}

	\begin{remark}\label{htMinimum}
		Due to Theorem \ref{densitystatementANDextension},  we know that $\inf_{\wmns} I_s=0$. Then, from the previous corollary applied with $\lambda=1$, we get that $h_T$ is a minimum point of $I_s$, thereby a critical point.
	\end{remark}

	Next, we obtain the following  general transformation law for  $I_{s,\lambda}$ under the action of $\mG$ which provides useful insight into the behavior of the functional. In the following, for $u\in W_{\mu_{N,s}}(\mathbb{R}^N)$ and $T\in\mathcal{G}$, we let $u_T$ be defined as in Theorem \ref{Invariance-MinimumIs}.

	\begin{remark}\label{remarkuT}
		By direct computation, for every $T\in\mathcal G$ there exists a
		constant $C_T>0$ such that
		$$
		\frac{1}{C_T}
		\leq
		\frac{\mu_{N,s}\circ T^{-1}}{\mu_{N,s}}
		\leq C_T
		\qquad \text{a.e. in }\mathbb R^N.
		$$
		Moreover, if $s>0$, then
		$$
		\frac{1}{C_T}
		\leq
		\frac{|\nabla(g_{N,s}\circ T^{-1})|}
		{|\nabla g_{N,s}|}
		\leq C_T
		\qquad \text{a.e. in }\mathbb R^N.
		$$
		If $s=0$, the latter estimate holds outside a sufficiently large
		ball. Consequently, $u\circ T\in
		W_{\mu_{N,s}}(\mathbb R^N)$ for every $u\in W_{\mu_{N,s}}(\mathbb R^N)$.
		Together with Lemma~\ref{prophT}, this also yields
		$u_T\in W_{\mu_{N,s}}(\mathbb R^N)$.
	\end{remark}

	\begin{proposition}\label{VarianceFunctionalLambdas}
		Let $u\in \wmns$, then the following equality holds for any $\lambda>0$, $T\in \mG$:
		\begin{align*}
			I_{s,\lambda}(u_T)=&I_{s,\lambda}(u)+ (\lambda-1) \int_{\mathbb{R}^N}\left(u_T-u\right) \, d\mu_{N,s} .
		\end{align*}
	\end{proposition}
	\begin{proof}
		
		Since 
		$$
		e^{u_T}\mu_{N,s} = e^{u\circ T} e^{h_T}  \mu_{N,s} =[T]^N e^{u\circ T} \mu_{N,s}\circ T,
		$$
		by making the change of variable $Tx=y$, we have
		\begin{align}\label{inteu}
			\int_{\mathbb{R}^N}e^{u_T}\mu_{N,s}\, dx=	\int_{\mathbb{R}^N}e^{u}\mu_{N,s}\, dy.
		\end{align}
		Similarly, since
		\begin{align*}
			\int_{\mathbb{R}^N} u(Tx)  \, d\mu_{N,s} = \frac{1}{[T]^N}\int_{\mathbb{R}^N} u(y)  \mu_{N,s}(T^{-1}y) \, dy = \int_{\R^N} u e^{h_{T^{-1}}} d\mu_{N,s},
		\end{align*}
		we have		
		\begin{align}\label{intut}
			\int_{\mathbb{R}^N}u_T \, d\mu_{N,s}= \int_{\mathbb{R}^N} h_T \, d\mu_{N,s} +\int_{\mathbb{R}^N} u e^{h_{T^{-1}}} \, d\mu_{N,s}.
		\end{align}
		Moreover, let us show the following claim 
		\begin{equation}\label{estimateI1s}\begin{aligned}
				\frac{1}{\widetilde{\omega_{N,s}}}\int_{\mathbb{R}^N}H_{N,s}(u_T) \, dx=&		
				\frac{1}{\widetilde{\omega_{N,s}}}  \int_{\mathbb{R}^N}\bigg( H_{N,s}(u)-H_{N,s}(h_{T^{-1}})\bigg)\, dx\\&
				-\left(\mathcal{J}_{N,s,T}+\mathcal{J}_{N,s,T^{-1}}\right)
				+ \int_{\mathbb{R}^N}u\, d\mu_{N,s} -\int_{\mathbb{R}^N}u e^{h_{T^{-1}}}\, d\mu_{N,s}.
			\end{aligned}
		\end{equation}
		In the following computations we denote by $g_{N,s}^T :=g_{N,s}\circ T$ and $u^T := u\circ T$, so that $u_T = u^T+h_T$. By definition of $H_{N,s}$ and using that $|\nabla g_{N,s}|^{N-2}\nabla g_{N,s} \cdot \nabla h_T\in L^1(\R^N,dx)$,  we have 
		\begin{equation}\label{part1}
			\begin{aligned}
				\frac{1}{\widetilde{\omega_{N,s}}}\int_{\mathbb{R}^N}H_{N,s}(u_T) \, dx=&\frac{1}{\widetilde{\omega_{N,s}}}  \int_{\mathbb{R}^N}\biggl( |\nabla (u^T+g_{N,s}^T)|^N \\&  -|\nabla g_{N,s}|^N - N|\nabla g_{N,s}|^{N-2}\nabla g_{N,s}\cdot\nabla (u^T+h_T)\biggr) \, dx \\
				=&\frac{1}{\widetilde{\omega_{N,s}}}  \int_{\mathbb{R}^N}\biggl( |\nabla (u^T+g_{N,s}^T)|^N \\&  -|\nabla g_{N,s}|^N - N|\nabla g_{N,s}|^{N-2}\nabla g_{N,s}\cdot\nabla u^T\biggr) \, dx-\J.
			\end{aligned}
		\end{equation}
		By using again the change of variable $Tx=y$ and the definition of $H_{N,s}$, we  get
		\begin{equation}\label{part2}
			\begin{aligned}
				\frac{1}{\widetilde{\omega_{N,s}}}  \int_{\mathbb{R}^N}&\biggl( |\nabla (u^T+g_{N,s}^T)|^N   -|\nabla g_{N,s}|^N - N|\nabla g_{N,s}|^{N-2}\nabla g_{N,s}\cdot\nabla u^T\biggr) \, dx	\\
				=& \frac{1}{\widetilde{\omega_{N,s}}}  \int_{\mathbb{R}^N}\bigg(|\nabla(u+g_{N,s})|^N-|\nabla g_{N,s}^{T^{-1}}|^N - N\left|\nabla  g_{N,s}^{T^{-1}}\right|^{N-2}\nabla g_{N,s}^{T^{-1}}\cdot \nabla u \bigg)\, dy
				\\ = &  	\frac{1}{\widetilde{\omega_{N,s}}}\int_{\mathbb{R}^N}H_{N,s}(u) \, dx + 	\frac{1}{\widetilde{\omega_{N,s}}}\int_{\R^N}\left( |\nabla g_{N,s}|^N - |\nabla g_{N,s}^{T^{-1}}|^N\right) dx\\
				&+\frac{N}{\widetilde{\omega_{N,s}}} \int_{\R^N} \left[ \, \left|\nabla  g_{N,s}\right|^{N-2}\nabla g_{N,s}- \left|\nabla  g_{N,s}^{T^{-1}}\right|^{N-2}\nabla g_{N,s}^{T^{-1}} \, \right]\cdot \nabla u \, dx.
			\end{aligned}
		\end{equation}
		Now, we observe that 
		\begin{equation*}
			\nabla g_{N,s}^{T^{-1}}  = \nabla h_{T^{-1}} + \nabla g_{N,s}, 
		\end{equation*}
		so that, using the definition of $H_{N,s}$
		$$
		\begin{aligned}
			|\nabla g_{N,s}|^N- |\nabla g_{N,s}^{T^{-1}}|^{N} &=  	|\nabla g_{N,s}|^N- | \nabla h_{T^{-1}} + \nabla g_{N,s}|^{N} \\
			& = - H_{N,s}(h_{T^{-1}}) - N |\nabla g_{N,s}|^{N-2}\nabla g_{N,s} \cdot \nabla h_{T^{-1}}, 
		\end{aligned}
		$$
		and 
		\begin{equation}\label{part3}
			\frac{1}{\widetilde{\omega_{N,s}}}\int_{\R^N}\left( |\nabla g_{N,s}|^N - |\nabla g_{N,s}^{T^{-1}}|^N\right) dx = -  \frac{1}{\widetilde{\omega_{N,s}}}\int_{\R^N} H_{N,s}(h_{T^{-1}}) \,dx - \mathcal{J}_{N,s,T^{-1}}.
		\end{equation}
		Similarly, we have
		\begin{equation}\label{part4}\begin{aligned}
				&\frac{N}{\widetilde{\omega_{N,s}}} \int_{\R^N} \left(\left|\nabla  g_{N,s}\right|^{N-2}\nabla g_{N,s}- \left|\nabla  g_{N,s}^{T^{-1}}\right|^{N-2}\nabla g_{N,s}^{T^{-1}}\right)\cdot \nabla u\, dx \\
				& = 	\frac{N}{\widetilde{\omega_{N,s}}} \int_{\R^N} \left[ \, \left|\nabla  g_{N,s}\right|^{N-2}\nabla g_{N,s}- \left| \nabla g_{N,s} +  \nabla h_{T^{-1}} \right|^{N-2} \left( \nabla g_{N,s} +  \nabla h_{T^{-1}} \right) \, \right]\cdot \nabla u\, dx .
			\end{aligned}
		\end{equation}
		In addition, by  Lemma \ref{prophT} and Remark \ref{htMinimum} we know that  $h_{T^{-1}}$ belongs to $\wmns$ and is a critical point of $I_{s}$. Then, applying \eqref{criticalpoints} and \eqref{fundamentalid}, we find:
		\begin{equation}\label{FinalStepconvergence}
			\begin{aligned}
				&	\frac{N}{\widetilde{\omega_{N,s}}} \int_{\R^N} \left(\left|\nabla  g_{N,s}\right|^{N-2}\nabla g_{N,s}- \left|  \nabla( h_{T^{-1}} +  g_{N,s})\right|^{N-2} \nabla (h_{T^{-1}} +  g_{N,s})\right)\cdot \nabla u\, dy
				\\&=\int_{\mathbb{R}^N}u\, d\mu_{N,s}-\int_{\mathbb{R}^N}ue^{h_{T^{-1}}}\, d\mu_{N,s}.
			\end{aligned}
		\end{equation}
		Using \eqref{part1}, \eqref{part2}, \eqref{part3}, \eqref{part4} and \eqref{FinalStepconvergence}, the claim  \eqref{estimateI1s} follows. Considering also \eqref{inteu} and \eqref{intut}, we find that
		\begin{equation}\begin{aligned}\label{almostthere}
				I_{s,\lambda}(u_T) =&	I_{s,\lambda}(u) + (\lambda - 1)\int_{\R^N} (u_T - u) \, d\mu_{N,s}
				\\ &  - \frac{1}{\widetilde{\omega_{N,s}}}  \int_{\mathbb{R}^N}H_{N,s}(h_{T^{-1}})\, dx + \int_{\mathbb{R}^N} h_{T} \, d\mu_{N,s}  -\mathcal{J}_{N,s,T}-\mathcal{J}_{N,s,T^{-1}}.
			\end{aligned}
		\end{equation}
		Additionally, note that by \eqref{integralht}
		\begin{align*}
			\int_{\mathbb{R}^N}h_T\, d\mu_{N,s}+\int_{\mathbb{R}^N} h_{T^{-1}} \, d\mu_{N,s}
			=& \J+\mathcal{J}_{N,s,T^{-1}},
		\end{align*}
		and we know by Proposition \ref{functionalwithu=0} that 
		$$
		\frac{1}{\widetilde{\omega_{N,s}}}  \int_{\mathbb{R}^N}H_{N,s}(h_{T^{-1}})\, dx+ \int_{\R^N} h_{T^{-1}} d\mu_{N,s}= 0.
		$$
		Coming back to \eqref{almostthere}, we deduce 
		$$ \dis 	I_{s,\lambda}(u_T) =	I_{s,\lambda}(u) + (\lambda - 1)\int_{\R^N} (u_T - u) \, d\mu_{N,s}.$$
	\end{proof}
	
	\begin{remark}
		If $T:\R^N \longrightarrow \R^N$ is a general conformal diffeomorphism such that $T(0)=0$ when $s>0$, then there exists an isometry $R:\R^N\longrightarrow \R^N $ fixing the origin, and $\bar T \in \mG$ such that $T= R\circ \bar T$. Since $g_{N,s}$ is radially symmetric, it is clear that  $g_{N,s} \circ R = g_{N,s}$, by which we get 
		$$
		I_s(u) = I_s(u\circ R)
		$$
		for any $u\in \wmns$. In addition, we have
		$$ u_T = u \circ R \circ { \bar T} + h_{\bar T}  = (u\circ R)_{\bar T}.
		$$ 
		and so  
		$$I_s(u_{T}) = I_s((u\circ R)_{\bar T} )= I_s(u\circ R )=I_s (u) .$$
	\end{remark}

	\begin{remark}\label{remarknew}
		There are several equivalent ways to state the identity in the previous proposition. For example, by using the definition of $u_T$ together with Proposition \ref{functionalwithu=0}, we find 
		$$
		\begin{aligned}
			I_{s,\lambda}(u_T) & = I_{s,\lambda}(u)+ (\lambda - 1)\left( \J- \frac{N-s}{N-1} \ln [T] + \int_{\R^N} \left(u(T x) -u(x)\right)\, d\mu_{N,s}(x)  \right) \\
			& = I_{s,\lambda}(u) + I_{s,\lambda}(h_T) + (\lambda - 1) \int_{\R^N} \left(u(T x) -u(x)\right)\, d\mu_{N,s}(x) .
		\end{aligned}
		$$
	\end{remark}

	We conclude this section by explicitly computing  the quantity $\mathcal{J}_{N,s,T}$ when $T = T_{t,0}$ is a simple dilation of parameter $t>0$.

	\begin{lemma}\label{L1product} For $s \in [0,N)$ and  $t>0$, $t\ne 1$ consider the map $T(x)= t x$.  Then we have
		\begin{align*}
			\mathcal{J}_{N,s,T}=
			\frac{N(-1)^{N-1}}{\left(b-1\right)^{N-1}}\ln b -  \frac{N}{\left(b-1\right)^{N-1}}\sum_{j=1}^{N-1}(-1)^{N+j} \frac{\left(b-1\right)^{j}}{j},
		\end{align*}
		where $b$ denotes $t^{\frac{N-s}{N-1}}$.
	\end{lemma}
	\begin{proof} Observe that 
		$$
		\begin{aligned}
			|\nabla g_{N,s}|^{N-2} \nabla g_{N,s} \cdot \nabla h_T
			=
			\left(\frac{N(N-s)}{N-1}\right)^N
			\frac{
				\left(t^{\frac{N-s}{N-1}}-1\right)
				|x|^{\frac{N(1-s)}{N-1}}
			}{
				\left(1+|x|^{\frac{N-s}{N-1}}\right)^N
				\left(1+t^{\frac{N-s}{N-1}}
				|x|^{\frac{N-s}{N-1}}\right)
			}.
		\end{aligned}
		$$
		Integrating in spherical coordinates and writing $z = |x|^\frac{N-s}{N-1}$, we find:
		\begin{align*}
			\J =
			N(b-1)\int_{0}^{+\infty} \frac{z^{N-1}}{(1+z)^N(1+b z)}\,dz .
		\end{align*}
		We make the change of variable $\rho = \frac{(b-1)z}{1+z}$ to find
		$$
		\begin{aligned}
			\J = \frac{N}{(b-1)^{N-1}}\int_0^{b-1} \frac{\rho^{N-1}}{\rho+1} d\rho\,.
		\end{aligned}
		$$
		Since 
		$$
		\frac{\rho^{N-1}}{\rho+1} = (-1)^{N-1}\left(\frac{(-\rho)^{N-1}-1}{\rho+1}+\frac{1}{1+\rho}\right) =(-1)^{N-1}\left( \frac{1}{1+\rho} - \sum_{k=0}^{N-2} (-1)^k \rho^k\,
		\right)\,,$$
		we get the result
		$$
		\begin{aligned}
			\mathcal{J}_{N,s,T} & =\frac{N(-1)^{N-1}}{(b-1)^{N-1}}\left(\ln b - \sum_{k=0}^{N-2}(-1)^k \frac{(b-1)^{k+1}}{k+1}\right)\, \\
			& = \frac{N(-1)^{N-1}}{(b-1)^{N-1}}\ln b - \frac{N}{(b-1)^{N-1}}\sum_{j=1}^{N-1}(-1)^{N+j} \frac{(b-1)^{j}}{j}\,.
		\end{aligned}
		$$
	\end{proof}

	\begin{remark}\label{RemarkIntegral} In Lemma \ref{L1product}, the parameters $t$ and $b$ are positive, and $\lim_{t \to 0} \mathcal{J}_{N,s,T} = -\infty$. Note that as $t\to 0$, we have $|\nabla g_{N,s}|^{N-2}\nabla g_{N,s} \cdot \nabla h_t \to - |\nabla g_{N,s}|^{N}$ in $L^1_{loc}(\R^N,dx)$. For any $r>0$, using the same change of variable, one can prove that
		$$
		\frac{1}{\widetilde{\omega_{N,s}}} \int_{B_r}|\nabla g_{N,s}|^N \, dx=\int_{0}^{r^{\frac{N-s}{N-1}}}\frac{z^{N-1}\, dz}{\left(1+z\right)^{N}} =  \ln(1+r^\frac{N-s}{N-1})-  \sum_{j=1}^{N-1} \frac{1}{j} \left( \frac{r^\frac{N-s}{N-1}}{1+r^\frac{N-s}{N-1}}\right)^j.
		$$
		We omit the explicit computation, since the integral was also computed in \cite{BorCinMan-2025} with different techniques. 
		
	\end{remark}
	
	\begin{remark}\label{asymptJ}
		As a consequence of Lemma \ref{L1product}, 	for $T=T_{t,0}$, we have
		$$
		\mathcal{J}_{N,s,T} \to \frac{N}{N-1} \text{ as } t\to +\infty,
		$$
		and 
		$$
		\mathcal{J}_{N,s,T} = \frac{N(N-s)}{N-1} \ln t + N \sum_{j=1}^{N-1}\frac{1}{j} + o(1) \text{ as } t\to 0^+.
		$$
	\end{remark}
	
	\begin{proof}[\textbf{Proof of Theorem \ref{Invariance-MinimumIs}}]
		Point \ref{minimumht} follows directly from Remark \ref{htMinimum}. The invariance of $I_s$ under the family $u_T$ is also deduced by Proposition \ref{VarianceFunctionalLambdas} with $\lambda=1$. Concerning \ref{allcritical}, from \cite{CirEspLi-2026} we know that all the weak solutions of \eqref{LiovilleNoIntegral}, have the form
		$$w(x)=g_{N,s}(Tx)+(N-s) \ln [T],$$
		for some $T \in \mathcal{G}$. Then, if $u\in \wmns$ is a critical point of $I_s$, by Lemma \ref{ConectionLioville} we know that $w:=u+g_{N,s}$ is a weak solution of \eqref{LiovilleWithIntegral}. Setting $c:=\ln \int_{\R^N} e^w/|x|^s\,dx $, the function $w-c$ is a weak solution to \eqref{LiovilleNoIntegral}. Therefore, $u$ can be expressed as $h_T+c$ for some $T\in \mG$, $c \in \R$.
	\end{proof}
	
	\begin{lemma}\label{ComparisionFunctionals}
		Assume $0< \lambda < 1$. For any $u\in W_{\mu_{N,s}}(\mathbb{R}^N)$, we have  $I_{s}(u)\leq I_{s,\lambda}(u),$ and equality holds if and only if $u$ is a constant function. Moreover, $\inf_{ \wmns}I_{s,\lambda}=0$.
	\end{lemma}
	\begin{proof}
		For any $u\in W_{\mu_{N,s}}(\mathbb{R}^N)$,  Jensen's inequality implies that 
		\[0\geq \int_{\mathbb{R}^N} u \,d\mu_{N,s}-\ln \left(\int_{\mathbb{R}^N} e^u \,d\mu_{N,s}\right),\]
		which yields $I_{s}(u)\leq I_{s,\lambda}(u)$ whenever $0< \lambda < 1$.  Moreover, since the exponential is a strictly convex function, the equality holds if and only if $u$ is a constant function (see for instance \cite[Theorem 2.2]{LiebLoss}). Finally, considering that $\inf_{u\in \wmns}I_s=0$ and $I_{s,\lambda}(u)=0$ whenever $u$ is a constant function, we deduce $\inf_{u\in \wmns}I_{s,\lambda}=0$.
		
	\end{proof}
	

	\begin{proof}[\textbf{Proof of Theorem \ref{InfLambdaLess1Bigger1}}]
		If $0 < \lambda <1$,	the point \ref{lambda<1} is obtained by using the previous lemma.\\
		If $\lambda >1$, part \ref{lambda>1} is a consequence of applying Corollary \ref{CorImportant} with $T= T_{t,0}$ and Remark \ref{asymptJ} to get 	\[I_{s,\lambda}(h_T)\rightarrow -\infty, \]
		as $t\to +\infty$ (or $t\to 0^+$). 
	\end{proof}
	
	\noindent
	As a consequence of Theorem \ref{InfLambdaLess1Bigger1}, we deduce the optimality of the constant $1/\widetilde{\omega_{N,s}}$ in the inequality \eqref{MT-SingularityNewSpace}. Indeed, if there were a real number $\lambda > 1$ such that
	\begin{equation*}
		\displaystyle \ln\left( \int_{\mathbb{R}^N} e^u \, d\mu_{N,s} \right)\leq \frac{1}{ \lambda \widetilde{\omega_{N,s}}}\int_{\mathbb{R}^N} H_{N,s}(u)\, dx + \int_{\mathbb{R}^N} u \, d\mu_{N,s}
	\end{equation*}
	for any function $u\in W_{\mu_{N,s}}(\mathbb{R}^N)$, it would follow that
	$$ \dis  \inf_{ u \in \wmns} I_{s,\lambda}(u) \geq 0.$$
	This is a contradiction to the fact that $\inf_{\wmns} I_{s,\lambda} = - \infty$ whenever $\lambda >1$.
	
	\section{Optimal constant and Poincaré Inequality}\label{SectionOptimalConstant}
	
	In this Section,  we obtain an alternative proof of the optimality of the constant $1/\widetilde{\omega_{N,s}}$ in the inequality \eqref{MT-SingularityNewSpace} via Proposition \ref{propN}. Moreover, we extend the Poincaré inequality also to $\wmns$. We start by dealing with the following problem in $\mathbb{R}^N$:
	\begin{equation}\label{Ec.LemN}
		-	\text{div} \left(|\nabla g_{N,s}|^{N-2}\nabla \phi+(N-2)|\nabla g_{N,s}|^{N-4}(\nabla g_{N,s}\cdot \nabla \phi)\nabla g_{N,s}\right)	=\frac{\widetilde{\omega_{N,s}}}{N}\phi\mu_{N,s},
	\end{equation}
	where $N\geq 2$ and $s \in [0,N)$. Concerning the concept of solution, we deal with the following notion.
	\begin{definition}
		A function $\phi\in W_{\mu_{N,s}}(\mathbb{R}^N)$ is a weak solution of \eqref{Ec.LemN} if it satisfies
		\begin{equation}\label{Ec2.LemN}
			\begin{aligned}
				\int_{\mathbb{R}^N}& \left(|\nabla g_{N,s}|^{N-2}\nabla \phi+(N-2)|\nabla g_{N,s}|^{N-4}(\nabla g_{N,s} \cdot \nabla \phi)\nabla g_{N,s}\right)\cdot \nabla \rho\, dx\\&=\frac{\widetilde{\omega_{N,s}}}{N}\int_{\mathbb{R}^N} \phi\rho\, d\mu_{N,s}
			\end{aligned}
		\end{equation}
		for all $\rho\in C^{\infty}_0(\mathbb{R}^N)$.
	\end{definition}
	
	\begin{lemma}\label{existencesolutionN}
		Assume $N\geq 2$ and $s \in [0,N)$, and let $\phi_{N,s}$ be defined as in \eqref{DefPhiNs}. Then:
		\begin{enumerate}[label=$\alph*$), ref=$\alph*$)]
			\item  $\phi_{N,s}\in \wmns$\label{a)}.
			\item  $\phi_{N,s}$ solves the problem \eqref{Ec.LemN} a.e. in $\mathbb{R}^N$.\label{b)}
			\item $\phi_{N,s}$ is a weak solution of \eqref{Ec.LemN}.\label{c)}
			\item $\int_{\mathbb{R}^N}\phi_{N,s} \, d\mu_{N,s} =0.$\label{d)}
		\end{enumerate}
	\end{lemma}
	\begin{proof}
		We first observe that
		\begin{align*}\phi_{N,s}(x)\mu_{N,s}(x)\sim |x|^{\frac{s-N^2}{N-1}}, \quad \text{ as $|x|\rightarrow +\infty$;}
		\end{align*}
		\begin{align*}\hspace{-0.9cm}\phi_{N,s}(x)\mu_{N,s}(x)\sim |x|^{-s}, \quad \text{ as $|x|\rightarrow 0$.}
		\end{align*}
		Since  $s \in [0,N)$ implies $(s-N^2)/(N-1)<-N$, we deduce that $ \phi_{N,s}\in L^1(\mathbb{R}^N,d \mu_{N,s})$.\\
		Now, taking into account Remark \ref{proprietagradvnslnxs}, we have
		\begin{equation}\label{nablaphi}
			\nabla \phi_{N,s}(x)= - N\left(\frac{N-s}{N-1}\right)^2 \frac{|x|^{\frac{1-s}{N-1}}}{(1+|x|^{\frac{N-s}{N-1}})^2} \frac{x}{|x|} = \frac{N-s}{N-1} \frac{1}{1+|x|^{\frac{N-s}{N-1}}} \nabla g_{N,s}(x),
		\end{equation}
		thus
		$ \left|\nabla \phi_{N,s}\right| \in L^{N}(\mathbb{R}^N,dx)$ and $|\nabla \phi_{N,s}|^2|\nabla g_{N,s}|^{N-2} \in L^{1}(\mathbb{R}^N,dx)$. As a consequence, \ref{a)} holds.
		
		To prove  statement \ref{b)}, observe that 
		\[\phi_{N,s}(x)=\left.\frac{\partial w_t(x)}{\partial t} \right|_{t=1},\]
		where $w_t(x)=g_{N,s}(tx)+(N-s)\ln t$ for $t>0$ are weak solutions to \eqref{LiovilleNoIntegral}, according to \cite{CirEspLi-2026}. Moreover, due to the regularity of $w_t$ in $\mathbb{R}^N\setminus\{0\}$,  almost everywhere we have
		\[\left.\frac{\partial}{\partial t} \left( -\Delta_N w_t - \frac{\widetilde{\omega_{N,s}}}{N}\frac{e^{w_t}}{|x|^s}
		\right) \right|_{t=1}=0,\]
		and so 
		\[\left.\frac{\partial}{\partial t}(-\Delta_N w_t)\right|_{t=1}=\frac{\widetilde{\omega_{N,s}}}{N}\mu_{N,s}\phi_{N,s}.\]
		On the other hand, we have
		\begin{align*}
			\left.\frac{\partial}{\partial t}(\Delta_N w_t)\right|_{t=1}=&\left.\text{div}\left((N-2)|\nabla w_t|^{N-4}\left(\nabla w_t \cdot \frac{\partial }{\partial t}(\nabla w_t)\right) \nabla w_t+ \frac{\partial }{\partial t} (\nabla w_t)|\nabla w_t|^{N-2}\right)\right|_{t=1}
			\\	=&\left.\text{div}\left((N-2)|\nabla w_t|^{N-4}\left(\nabla w_t \cdot \nabla\left(\frac{\partial }{\partial t}w_t\right)\right) \nabla w_t +\nabla\left(\frac{\partial }{\partial t} w_t\right)|\nabla w_t|^{N-2}\right)\right|_{t=1}
			\\	= &\, \text{div} \left(|\nabla g_{N,s}|^{N-2}\nabla \phi_{N,s}+(N-2)|\nabla g_{N,s}|^{N-4}(\nabla g_{N,s}\cdot \nabla \phi_{N,s})\nabla g_{N,s}\right),
		\end{align*}
		where we have used that the divergence and the gradient commute with $\partial/\partial t$. Therefore $\phi_{N,s}$ solves \eqref{Ec.LemN} a.e. in $\R^N$. 
		
		We now prove statement \ref{c)}. For any $ r>1 $, considering the regularity of $g_{N,s}$ and $\phi_{N,s}$ in $ B_{r} \setminus B_{1/r}$, Gauss-Green formula yields 
		\begin{equation*}
			\begin{aligned}
				&-\int_{B_r\setminus B_{1/r}} \text{div} \biggl[ \, |\nabla g_{N,s}|^{N-2}\nabla \phi_{N,s}+(N-2)|\nabla g_{N,s}|^{N-4}(\nabla g_{N,s}\cdot \nabla \phi_{N,s})\nabla g_{N,s} \biggr] \rho \, dx  \\
				=&\int_{B_r\setminus B_{1/r}} \left(|\nabla g_{N,s}|^{N-2}\nabla \phi_{N,s}+(N-2)|\nabla g_{N,s}|^{N-4}(\nabla g_{N,s} \cdot \nabla\phi_{N,s})\nabla g_{N,s}\right) \cdot  \nabla \rho \, dx\\
				& -\int_{\partial \left( B_r \setminus B_{1/r} \right)} \rho \left(|\nabla g_{N,s}|^{N-2}\nabla \phi_{N,s}+(N-2)|\nabla g_{N,s}|^{N-4}(\nabla g_{N,s}\cdot \nabla\phi_{N,s})\nabla g_{N,s}\right) \cdot \nu \, d \sigma
			\end{aligned}
		\end{equation*}
		for every $\rho\in C^{\infty}_0(\mathbb{R}^N)$,	where $\nu$ is the outward unit vector to $\partial \left(B_r \setminus B_{1/r} \right)$.
		
		In particular, we can rewrite the boundary integral as
		\begin{align*}
			\dis 
			\int_{\partial \left( B_r \setminus B_{1/r} \right)}& \rho \left(|\nabla g_{N,s}|^{N-2}\nabla \phi_{N,s}+(N-2)|\nabla g_{N,s}|^{N-4}(\nabla g_{N,s}\cdot \nabla\phi_{N,s})\nabla g_{N,s}\right) \cdot \nu \, d \sigma \\
			= & \int_{\partial  B_r } \rho \left(|\nabla g_{N,s}|^{N-2}\nabla \phi_{N,s}+(N-2)|\nabla g_{N,s}|^{N-4}(\nabla g_{N,s}\cdot \nabla\phi_{N,s})\nabla g_{N,s}\right) \cdot \nu \, d \sigma \\
			\quad &- \int_{\partial B_{1/r} } \rho \left(|\nabla g_{N,s}|^{N-2}\nabla \phi_{N,s}+(N-2)|\nabla g_{N,s}|^{N-4}(\nabla g_{N,s}\cdot \nabla\phi_{N,s})\nabla g_{N,s}\right) \cdot \nu \, d \sigma.
		\end{align*}
		Concerning   $\partial B_r$, we get
		\begin{align*}
			& \left| \int_{\partial  B_r } \rho \left(|\nabla g_{N,s}|^{N-2}\nabla \phi_{N,s}+(N-2)|\nabla g_{N,s}|^{N-4}(\nabla g_{N,s}\cdot \nabla\phi_{N,s})\nabla g_{N,s}\right) \cdot \nu \, d \sigma   \right| \\
			& \quad \leq 	(N-1) \lVert \rho \rVert_{L^{\infty}(\mathbb{R}^N)} \int_{\partial  B_r } |\nabla g_{N,s}|^{N-2} \left| \nabla \phi_{N,s}  \right| \, d \sigma \\
			& \qquad =  C_{N,s}	\int_{\partial B_r } \frac{|x|^{1-s}}{(1+|x|^{\frac{N-s}{N-1}})^N} \, d \sigma
			=C_{N,s}\omega_{N-1}\frac{r^{N-s}}{(1+r^{\frac{N-s}{N-1}})^N}\rightarrow 0 \quad \text{ as $r \to +\infty$.}
		\end{align*}
		Analogously, the integral over the boundary $\partial B_{1/r}$ converges to zero.\\
		Now, observe that \ref{b)} also implies that
$$ - \text{div} \biggl[ \, |\nabla g_{N,s}|^{N-2}\nabla \phi_{N,s}+(N-2)|\nabla g_{N,s}|^{N-4}(\nabla g_{N,s}\cdot \nabla \phi_{N,s})\nabla g_{N,s} \biggr] \rho  = \frac{\widetilde{\omega_{N,s}}}{N}	\phi_{N,s} \rho \mu_{N,s} $$
in $\mathbb{R}^N \setminus\{0\}$ and $\phi_{N,s} \rho \in L^{1}( \mathbb{R}^N,  d \mu_{N,s}).$
On the other hand,
		\begin{align*}
			\dis  & \bigg| \, \left(|\nabla g_{N,s}|^{N-2}\nabla \phi_{N,s}+(N-2)|\nabla g_{N,s}|^{N-4}(\nabla g_{N,s} \cdot \nabla\phi_{N,s})\nabla g_{N,s}\right) \cdot  \nabla \rho  \, \bigg|  \\
			&\quad \leq (N-1) |\nabla g_{N,s}|^{N-2} \left| \nabla \phi_{N,s}  \right| \left| \nabla \rho \right|,
		\end{align*}
		where $|\nabla g_{N,s}|^{N-2} \left| \nabla \phi_{N,s}  \right| \left| \nabla \rho \right| \in L^1(\mathbb{R}^N,dx)$ follows by $\phi_{N,s}\in\wmns$,  Remark \ref{proprietagradvnslnxs} and  H\"older inequality.  Therefore, we can apply the Dominated Convergence Theorem to conclude that 
		$\phi_{N,s}$ satisfies \eqref{Ec2.LemN} for every $\rho\in C^{\infty}_0(\mathbb{R}^N)$, and so is a weak solution of \eqref{Ec.LemN}.\\
		Finally, \ref{d)} is also achieved since 
		\begin{align*} 
			\int_{\mathbb{R}^N}\phi_{N,s}\mu_{N,s} \,dx=&  C_{N,s}(N-1)\int_0^{\infty}\frac{r^{N-1-s}}{(1+r^{\frac{N-s}{N-1}})^{N+1}} \, dr-C_{N,s}\int_0^{\infty}\frac{r^{\frac{N-s}{N-1}}r^{N-1-s}}{(1+r^{\frac{N-s}{N-1}})^{N+1}}\, dr
			\\=&C_{N,s}(N-1)\int_0^1(1-t)t^{N-2}\, dt-C_{N,s}\int_0^1 t^{N-1}\, dt= 0,
		\end{align*}
		where we have used the change of variable $r^{\frac{N-s}{N-1}}=t/(1-t)$.
	\end{proof}

	\begin{remark}\label{ExtEqtoWmuns}
		By density, note that $\phi_{N,s}$ satisfies \eqref{Ec2.LemN} for every $\rho\in \wmns$.
		
		In particular if we select $\phi_{N,s}$ as test function, we get
		\begin{equation}\label{eqphiN}
			\begin{aligned}
				\int_{\mathbb{R}^N} &|\nabla g_{N,s} |^{N-2}|\nabla \phi_{N,s}|^2+(N-2)|\nabla g_{N,s} |^{N-4}(\nabla g_{N,s} \cdot \nabla \phi_{N,s})^2\, dx \\
				&=\frac{\widetilde{\omega_{N,s}}}{N}\int_{\mathbb{R}^N} \mu_{N,s}\phi^2_{N,s}\, dx.
			\end{aligned}
		\end{equation}
		Furthermore, since both $g_{N,s}$ and $\phi_{N,s}$ are radially symmetric, we obtain
		\begin{equation}\label{SPECIALeqphiN}
			\begin{aligned}
				\int_{\mathbb{R}^N} \phi^2_{N,s}\, d\mu_{N,s} = \frac{N(N-1)}{\widetilde{\omega_{N,s}}}\int_{\mathbb{R}^N} |\nabla g_{N,s} |^{N-2}|\nabla \phi_{N,s}|^2 dx.
			\end{aligned}
		\end{equation}
	\end{remark}
	
	\begin{remark}
		A remarkable difference between the case $N= 2$ and $N>2$ is the absence of commutation between $\Delta_N$ and $\partial/\partial t$ due to the nonlinearity of the operator for $N>2$.
	\end{remark}

	\begin{proof}[\textbf{Proof of Proposition \ref{propN}}]
		
		Let us consider for $s \in [0,N)$ the quantity
		\[F_\lambda(u)=\frac{\frac{1}{\widetilde{\omega_{N,s}}}\int_{\mathbb{R}^N}H_{N,s}(u)\, dx+\lambda \int_{\mathbb{R}^N}u \, d\mu_{N,s} }{\ln \left(\int_{\mathbb{R}^N}e^u \, d\mu_{N,s}\right)}.\]
		Let us also consider the set 
		$$
		\W:=\left\{u\in \wmns \::\; \int_{\mathbb R^N} u \, d\mu_{N,s} >0\right\}, 
		$$
		so that 
		$$
		G_\lambda = \inf_{\W} F_\lambda.$$
		Note that $F_\lambda$ is well defined on $\W$ and in fact on a slightly larger set. Indeed, by Jensen's inequality $\ln (\int_{\mathbb{R}^N} e^u \, d\mu_{N,s})\geq \int_{\mathbb{R}^N} u \, d\mu_{N,s}$,  where the equality is only achieved for $u$ constant, so $F_{\lambda}(u)$ is well defined for every $u\in \wmns$ such that
		$$
		u\not\equiv 0 \quad \text{ and }\int_{\mathbb R^N} u \, d\mu_{N,s}\ge 0. 
		$$ 
		In particular, if $ u \in \W$, then $\ln (\int_{\mathbb{R}^N} e^u \, d\mu_{N,s})>0$ consequently, for any $\lambda>0$, $F_\lambda(u)\geq 0$, so   $G_\lambda$ is  well defined and $G_\lambda \geq 0$.\par 
		Due to the fact that positive constant functions belong to $\W$, we have that for $u\equiv d>0$,
		\[F_\lambda[d]= \lambda.\] 
		As a consequence, 
		\begin{equation}\label{Gleq}
			G_\lambda\leq \lambda.
		\end{equation}

		Let us observe that any $u\in \W$ can be written as $v+c$, with $c>0$ and  $v\in W_{\mu_{N,s}}(\mathbb{R}^N)$ satisfying that $\int_{\mathbb{R}^N}v \, d\mu_{N,s}=0$. Moreover, 
		\begin{equation*}
			F_\lambda (v+c)=\frac{\frac{1}{\widetilde{\omega_{N,s}}}\int_{\mathbb{R}^N}H_{N,s}(v)\, dx+\lambda \int_{\mathbb{R}^N}c \, d\mu_{N,s} }{c+\ln \left(\int_{\mathbb{R}^N}e^v \, d\mu_{N,s}\right)}.
		\end{equation*}
		Note that:
		\begin{enumerate}[label=$\arabic*$), ref=$\arabic*$)]	
			\item $\displaystyle\lim_{c\rightarrow 0} F_\lambda(v+c)=F_\lambda(v)$, provided $v\not\equiv 0$.\label{limitc}
			\item  From the fact that $\int_{\mathbb{R}^N}\mu_{N,s} \, dx=1$ and Theorem \ref{densitystatementANDextension},  
			\begin{equation}\label{EcAplicacion2}
				F_\lambda (v+c)\geq\frac{\ln \left(\int_{\mathbb{R}^N}e^v \, d\mu_{N,s}\right) +\lambda c}{c+\ln \left(\int_{\mathbb{R}^N}e^v \, d\mu_{N,s}\right)}.
			\end{equation}
		\end{enumerate}
		Let us take into account two cases:\par
		\medskip
		\noindent
		\textbf{Case 1.} $\lambda<1$. We can use \eqref{EcAplicacion2}, so 
		\[
		\begin{aligned}
			F_\lambda(v+c)& \geq \frac{\ln \left(\int_{\mathbb{R}^N}e^v \, d\mu_{N,s}\right) +\lambda c}{c+\ln \left(\int_{\mathbb{R}^N}e^v \, d\mu_{N,s}\right)} \\
			& = \lambda+(1-\lambda)\frac{\ln \left(\int_{\mathbb{R}^N}e^v \, d\mu_{N,s}\right)}{c+\ln \left(\int_{\mathbb{R}^N}e^v \,  d\mu_{N,s}\right)}\geq \lambda.
		\end{aligned}
		\]
		Then $F_{\lambda}(u)\ge \lambda$ for any  $u\in \W$ and  $G_\lambda \geq \lambda$. By \eqref{Gleq}, $G_\lambda =\lambda$ and the proof is concluded in this case.
		\medskip
		\par 
		\noindent
		\textbf{Case 2.} $\lambda\geq 1$.
		Again, by \eqref{EcAplicacion2}, we get $F_\lambda(u)\geq 1$ 
		for any $u\in \W$. Then,
		\begin{equation*}
			G_\lambda\geq 1.
		\end{equation*}
		
		We need to prove the reverse inequality. We may choose $v$ as $\varepsilon \phi_{N,s}$, with $\phi_{N,s}$ as in Lemma \ref{existencesolutionN} and $\varepsilon>0$. Then 
		\[F_\lambda(\varepsilon \phi_{N,s})=\frac{\frac{1}{\widetilde{\omega_{N,s}}}\int_{\mathbb{R}^N}H_{N,s}(\varepsilon\phi_{N,s})\, dx +\lambda \varepsilon \int_{\mathbb{R}^N}\phi_{N,s}\, d\mu_{N,s}}{\ln \left(\int_{\mathbb{R}^N}e^{\varepsilon\phi_{N,s}}\, d\mu_{N,s}\right)}.\]
		On the other hand, through Taylor expansion, considering that $\int_{\mathbb{R}^N}\phi_{N,s}d\mu_{N,s}=0$, we get
		\begin{equation}\label{TaylorExpansionExp}
			\begin{aligned}
				\ln \left(\int_{\mathbb{R}^N}e^{\varepsilon\phi_{N,s}} \, d\mu_{N,s}\right)=&\ln \left(1+\frac{1}{2}\int_{\mathbb{R}^N}\varepsilon^2\phi_{N,s}^2 \, d\mu_{N,s}+ o(\varepsilon^2)\right)\\
				&= \frac{\varepsilon^2}{2}\int_{\mathbb{R}^N}\phi_{N,s}^2 \, d\mu_{N,s}+o(\varepsilon^2).
			\end{aligned}
		\end{equation}
		Hence,
		\begin{align*}
			F_\lambda(\varepsilon \phi_{N,s})=&
			\frac{\frac{1}{\widetilde{\omega_{N,s}}}\int_{\mathbb{R}^N}H_{N,s}(\varepsilon\phi_{N,s})\, dx}{\frac{\varepsilon^2}{2}\int_{\mathbb{R}^N}\phi_{N,s}^2 \, d\mu_{N,s}+o(\varepsilon^2)}.
		\end{align*}
		The next step is to take into account that $\phi_{N,s}$ solves \eqref{eqphiN} and to achieve that
		\[F_\lambda(\varepsilon \phi_{N,s})\leq 1+o(1),\]
		as $\varepsilon\to 0$.  
		
		First, for $X,Y\in \R^N$, from Taylor's Theorem with Lagrange remainder for the function $f(t)=|X+tY|^N$,  
		there exists some $c\in (0,1)$ such that 
		\[f(1)= f(0)+f'(0)+\frac{f''(0)}{2}+\frac{f^{3)}(c)}{3!}.\]
		
		If $N=2$, $f^{3)} \equiv 0$. If $N\ge 3$, for $X=\nabla g_{N,s}$, $Y=\nabla \varepsilon \phi$,
		\begin{align*}
			|f^{3)}(c)|=&\left|N(N-2)(N-4)|\nabla g_{N,s}+c\varepsilon \nabla \phi_{N,s} |^{N-6}((c\varepsilon\nabla\phi_{N,s} +\nabla g_{N,s})\cdot \varepsilon\nabla \phi_{N,s})^3\right.\\&\left.+
			3N(N-2)|\nabla g_{N,s}+c\varepsilon\nabla\phi_{N,s}|^{N-4}((\nabla g_{N,s}+c\varepsilon\nabla\phi_{N,s})\cdot \varepsilon\nabla\phi_{N,s})|\varepsilon\nabla\phi_{N,s}|^2\right|
			\\\leq& C_N\varepsilon^3|\nabla \phi_{N,s}|^3 |\nabla g_{N,s}|^{N-3},
		\end{align*}
		then from Remark \ref{Nlaplaciano} and \eqref{nablaphi} the following estimates are deduced,
		\begin{align*}
			|\nabla \phi_{N,s}(x)|^3 |\nabla g_{N,s}(x)|^{N-3}&\sim |x|^{N\frac{1-s}{N-1}},\quad\text{ as $|x|\rightarrow 0$;}\\
			|\nabla \phi_{N,s}(x)|^3 |\nabla g_{N,s}(x)|^{N-3}&\sim |x|^{\frac{-2N+3s-N^2}{N-1}},\quad\text{ as $|x|\rightarrow +\infty$.}
		\end{align*}
		Since $N\frac{1-s}{N-1}>-N$ and  $\frac{-2N+3s-N^2}{N-1}<-N$ for $s<N$, there exists some $h\in L^1(\mathbb{R}^N,dx)$ such that
		\begin{equation}\label{estimatef^3}
			|f^{3)}(c)|\leq \varepsilon^3 h.
		\end{equation}
		
		Hence,
		\begin{align*}
			|\nabla g_{N,s}+\varepsilon \nabla \phi_{N,s}|^N=&|\nabla g_{N,s}|^N+N|\nabla g_{N,s}|^{N-2}\nabla g_{N,s}\cdot \varepsilon\nabla\phi_{N,s}\\& +\frac{N(N-2)}{2}|\nabla g_{N,s}|^{N-4}(\nabla g_{N,s}\cdot \varepsilon\nabla\phi_{N,s})^2 \\&+\frac{N}{2} |\nabla g_{N,s}|^{N-2}|\varepsilon\nabla\phi_{N,s}|^2 + \frac{f^{3)}(c)}{3!}.
		\end{align*}
		As a consequence,
		\begin{align}\label{TaylorHN}
			\int_{\mathbb{R}^N} H_{N,s}(\varepsilon\phi_{N,s})\, dx=&\frac{\varepsilon^2N}{2}\int_{\mathbb{R}^N}(N-2)|\nabla g_{N,s}|^{N-4}(\nabla g_{N,s}\cdot \nabla\phi_{N,s})^2  dx\\&+\frac{\varepsilon^2N}{2}\int_{\mathbb{R}^N}|\nabla g_{N,s}|^{N-2}|\nabla\phi_{N,s}|^2\, dx +\frac{1}{3!}\int_{\mathbb{R}^N}f^{3)}(c)\, dx. \nonumber
		\end{align}
		Then, using that $\phi_{N,s}$ solves \eqref{eqphiN}, 
		\begin{align*}
			\frac{1}{\widetilde{\omega_{N,s}}}\int_{\mathbb{R}^N} H_{N,s}(\varepsilon\phi_{N,s})\, dx=&\frac{\varepsilon^2}{2}\int_{\mathbb{R}^N}\phi_{N,s}^2 \mu_{N,s}\, dx +  \frac{1}{3!\, \widetilde{\omega_{N,s}}} \int_{\mathbb{R}^N}f^{3)}(c)\, dx
		\end{align*}
		so,
		\[F_\lambda(\varepsilon\phi_{N,s})=\frac{\frac{\varepsilon^2}{2}\int_{\mathbb{R}^N}\phi_{N,s}^2 \, d\mu_{N,s}	+ \frac{1}{3!\, \widetilde{\omega_{N,s}}}\int_{\mathbb{R}^N}f^{3)}(c)\, dx}{\frac{\varepsilon^2}{2}\int_{\mathbb{R}^N}\phi_{N,s}^2 \, d\mu_{N,s}+o(\varepsilon^2)}.\]	
		
		Additionally, due to \eqref{estimatef^3},
		\[F_\lambda(\varepsilon\phi_{N,s})\leq \frac{\frac{\varepsilon^2}{2}\int_{\mathbb{R}^N}\phi_{N,s}^2 \, d\mu_{N,s}	+\frac{\varepsilon^3}{3!\, \widetilde{\omega_{N,s}}}\int_{\mathbb{R}^N}h \, dx}{\frac{\varepsilon^2}{2}\int_{\mathbb{R}^N}\phi_{N,s}^2 \, d\mu_{N,s}+o(\varepsilon^2)}= \frac{\frac{\varepsilon^2}{2}\int_{\mathbb{R}^N}\phi_{N,s}^2 \, d\mu_{N,s}	+o(\varepsilon^2)}{\frac{\varepsilon^2}{2}\int_{\mathbb{R}^N}\phi_{N,s}^2 \, d\mu_{N,s}+o(\varepsilon^2)}\to 1,\]
		as $\varepsilon\to 0$.	Then, due to \ref{limitc}, we get
		$$
		\limsup_{\varepsilon \to 0^+} \lim_{c\to 0^+} F_\lambda ( \varepsilon \phi_{N,s}+c) = \limsup_{\varepsilon \to 0^+} F_\lambda(\varepsilon \phi_{N,s}) \le 1.
		$$
		Since for any $c,\varepsilon>0$, we have $\varepsilon \phi_{N,s}+c\in \W$, we conclude that  
		\begin{align*}
			G_\lambda = 1, \text{ for $\lambda\geq 1$}.
		\end{align*}
	\end{proof}
	
	As a consequence, the optimality of the constant $1/\widetilde{\omega_{N,s}}$ in \eqref{MT-SingularityNewSpace} is deduced. Indeed, if there were a real number $\lambda > 1$ such that
	\begin{equation*}
		\displaystyle \ln\left( \int_{\mathbb{R}^N} e^u \, d\mu_{N,s} \right)\leq \frac{1}{ \lambda \widetilde{\omega_{N,s}}}\int_{\mathbb{R}^N} H_{N,s}(u)\, dx + \int_{\mathbb{R}^N} u \, d\mu_{N,s}
	\end{equation*}
	for any function $u\in \W$, it would follow that
	$$ \dis   1 \leq \frac{1}{\lambda} G_{\lambda},$$
	which is a contradiction to the fact that $G_{\lambda}= 1$ whenever $\lambda >1$.

	Finally, we show that the following Poincar\'e type inequality, proved  in \cite{LamLu-2018} for smooth compact functions, can be also extended to our space $\wmns$.
	
	\begin{proof}[\textbf{Proof of Proposition \ref{Poincare}}]
		Let us consider $\varepsilon(u-\overline{u})$ with $\varepsilon>0$. Since $\varepsilon(u-\overline{u})\in W_{\mu_{N,s}}(\mathbb{R}^N)$, by Theorem \ref{densitystatementANDextension} we deduce
		\begin{align*}
			\lim_{\varepsilon\rightarrow 0}\left[\frac{1}{\varepsilon^2}\displaystyle \ln\left( \int_{\mathbb{R}^N} e^{\varepsilon(u-\overline{u})} \, d\mu_{N,s} \right)\right]\leq 	\lim_{\varepsilon\rightarrow 0}\left[\frac{1}{\varepsilon^2\widetilde{\omega_{N,s}}}\int_{\mathbb{R}^N} H_{N,s}(\varepsilon(u-\overline{u}))\, dx\right].
		\end{align*}
		Therefore, the result follows from using Taylor expansions as in \eqref{TaylorExpansionExp} and \eqref{TaylorHN}, that is
		\begin{align*}
			\lim_{\varepsilon\rightarrow 0}\left[\frac{1}{\varepsilon^2}\displaystyle \ln\left( \int_{\mathbb{R}^N} e^{\varepsilon(u-\overline{u})} \, d\mu_{N,s} \right)\right]=\frac{1}{2}\int_{\mathbb{R}^N}|u-\overline{u}|^2\, d\mu_{N,s},
		\end{align*}
		and 
		\begin{align*}
			\lim_{\varepsilon\rightarrow 0}\left[\frac{1}{\varepsilon^2\widetilde{\omega_{N,s}}}\int_{\mathbb{R}^N} H_{N,s}(\varepsilon(u-\overline{u}))\, dx\right]=&\frac{N(N-2)}{2\widetilde{\omega_{N,s}}}\int_{\mathbb{R}^N} |\nabla g_{N,s}|^{N-4}(\nabla g_{N,s}\cdot \nabla u )^2\\&+\frac{N}{2\widetilde{\omega_{N,s}}}\int_{\mathbb{R}^N}|\nabla g_{N,s}|^{N-2}|\nabla u|^2\, dx.
		\end{align*}
		Moreover, \eqref{SPECIALeqphiN} and  Lemma \ref{existencesolutionN} imply that the equality is achieved by the function $\phi_{N,s}$.
	\end{proof}
	\begin{remark}
		For $N=2$ we have the classical Poincar\'e inequality. Moreover, as a consequence of the previous proposition, any $u\in\wmns$ satisfies 
		
		\begin{align*}
			\int_{\mathbb{R}^N}|u-\overline{u}|^2\, d\mu_{N,s}\leq& \frac{N(N-1)}{\widetilde{\omega_{N,s}}} \int_{\mathbb{R}^N}|\nabla g_{N,s}|^{N-2}|\nabla u|^2\, dx.
		\end{align*}
	\end{remark}

	\section{Equivalence}\label{SectionEquivalence}
	In this section we obtain inequality \eqref{Chang-Carleson-Singularity} and establish the equivalence with \eqref{MT-SingularityNewSpace}. First, we recall that Carleson and Chang in \cite{CarCha-1986} obtained for any $u \in W_0^{1,N}(B_1)$ that
	\begin{equation}\label{InequalityCC}
		\ln \left(\frac{1}{|B_1|} \int_{B_1} e^u \, dx \right)< \frac{1}{\widetilde{\omega_{N}}}\int_{B_1} |\nabla u|^N + \sum_{k=1}^{N-1} \frac{1}{k}.  
	\end{equation}

	On the other hand, as a consequence of~\eqref{MTSingularity}, one obtains the following logarithmic version of the Moser--Trudinger inequality with a singularity. The proof is rather standard by means of Young's inequality.

	\begin{proposition}\label{logMTsingular}
		Let $N \geq 2$, $s\in [0,N)$ and let  $\Omega \subset \mathbb{R}^N$ be a bounded domain containing the origin. For any $u\in W^{1,N}_0(\Omega)$ we have
		\begin{equation}\label{logMTsingularinequality}
			\displaystyle \ln\left(\frac{1}{|\Omega|^{1-\frac{s}{N}}} \int_{\Omega} \frac{e^u }{|x|^s} \, dx \right) \leq \frac{1}{\widetilde{\omega_{N,s}}}\int_\Omega |\nabla u|^N \, dx  + \ln C_{N,s},
		\end{equation}
		
		\noindent
		where $C_{N,s}$ is the constant appearing in \eqref{MTSingularity}. 	
	\end{proposition}

	Now if we consider $\Omega=B_1$, we can improve inequality \eqref{logMTsingularinequality} by means of inequality \eqref{InequalityCC}. To see this, let us recall some results proved in \cite{LamLu-2017} (see also \cite{LamLu-2018}). 
	
	Let $\Omega \subset \mathbb{R}^N$ be a bounded domain, let $D:=\{x\in \mathbb{R}^N: |x|^{\frac{s}{N-s}}x\in \Omega \}$ and let  us consider the vector-valued functions $F_{N,s}:  D \to \Omega$ and $F_{N,s}^{-1}:  \Omega \to D$ defined by
	$$ \displaystyle F_{N,s}(x):= |x|^{\frac{s}{N-s}}x;$$
	$$ \displaystyle F_{N,s}^{-1}(x):= |x|^{-\frac{s}{N}}x.$$
	
	\noindent
	By a direct computation, we have
	$$ \displaystyle \det \left( J_{F_{N,s}}(x) \right) = \frac{N}{N-s} |x|^{\frac{Ns}{N-s}}. $$
	
	\noindent
	Now, let us consider the map $D_{N,s}$ and $D_{N,s}^{-1}$ that associate to a function $u: \Omega \to \mathbb{R}$  and $v: D \to \mathbb{R}$ the functions  $D_{N,s} u: D \to \mathbb{R}$ and $D_{N,s}^{-1} v: \Omega \to \mathbb{R}$ defined as
	$$D_{N,s} u(x):=\left(\frac{N-s}{N}\right)^{\frac{1}{N'}}u(F_{N,s}(x))=\left(\frac{N-s}{N}\right)^{\frac{1}{N'}}u(|x|^{\frac{s}{N-s}}x);$$
	$$D_{N,s}^{-1} v(x):=\left(\frac{N}{N-s}\right)^{\frac{1}{N'}}v(F_{N,s}^{-1}(x))=\left(\frac{N}{N-s}\right)^{\frac{1}{N'}}v(|x|^{-\frac{s}{N}}x).$$
	
	\noindent
	Let us recall the following result (see \cite[Lemma 2.2]{LamLu-2017} and \cite[Lemma 2.1]{DonLamLu-2018}).
	
	\begin{lemma}\label{LamLuLemma}
		
		\begin{enumerate}
			
			\item \label{Prop3}If $f$ is a continuous function, then
			\[\int_{\Omega}\frac{f \left(\left(\frac{N-s}{N}\right)^{1/N'}u(y)\right)}{|y|^s} \, dy =\frac{N}{N-s} \int_D f(D_{N,s}u (x)) \, dx;\]
			\[\int_{D} f(v(x)) \, dx =\frac{N-s}{N} \int_{ \Omega} \frac{f \left(\left(\frac{N-s}{N}\right)^{1/N'} D_{N,s}^{-1} v(y)\right)}{|y|^s}  \,  dy.\]

			\item \label{Prop4} If $\nabla u\in L^N(\Omega,dx)$, then $\nabla D_{N,s} u \in L^N(D,dx)$. Also if $\nabla D_{N,s}^{-1}v\in L^N(\Omega,dx)$, then $\nabla v\in L^N(D,dx)$. Moreover,  
			$$\int_D |\nabla D_{N,s} u(x)|^N \, dx \leq \int_{\Omega} |\nabla u(y)|^N \, dy;$$
			$$\int_D |\nabla v(x)|^N \, dx \leq \int_{\Omega} |\nabla  D_{N,s}^{-1} v(y)|^N \, dy.$$
			
			\noindent				
			The equality occurs if and only if $u$ is radially symmetric.
			
		\end{enumerate}
	\end{lemma}

	\begin{remark}\label{DNuSob}
		If $u\in W^{1,N}_0(\Omega)$, then $D_{N,s}(u)\in W^{1,N}_0(D)$.
		Indeed, taking into account the change of variables $y = F_{N,s}(x)$, we have
		\begin{align*}
			\displaystyle \int_D | D_{N,s} u(x)|^N \, dx
			& = \left(\frac{N-s}{N}\right)^{\frac{N}{N'}} \int_D |u(F_{N,s}(x))|^N \, dx \\
			& = \left(\frac{N-s}{N}\right)^{\frac{N}{N'}} \int_{\Omega} |u(y)|^N \frac{1}{\left| \det \left( J_{F_{N,s}}(x) \right) \right|}  \, dy \\
			& = \left(\frac{N-s}{N}\right)^{N} \int_{\Omega} |u(y)|^N \frac{1}{\left| y \right|^s}  \, dy.
		\end{align*}
		
		\noindent
		Observe that $u \in L^t(\Omega,dx)$ for any $t \geq 1$ and that for any $s \in [0, N)$ there exists $q>1$ such that $sq < N$. Hence, by H\"older inequality with exponents $q$ and $q'$,
		$$ \displaystyle \int_{\Omega} |u(y)|^N \frac{1}{\left| y \right|^s}  \, dy  \leq \lVert u \rVert_{Nq'}^N \left( \int_{\Omega} \frac{1}{\left| y \right|^{sq}}  \, dy \right)^{\frac{1}{q}} < + \infty.$$		
		Therefore, $ D_{N,s} u(x) \in L^N(D,dx)$.\\
		On the other hand, by (ii) of Lemma \ref{LamLuLemma} we have $ \left| \nabla D_{N,s} u \right| \in L^N(D,dx)$.\\ Consequently, $D_{N,s}(u)\in W^{1,N}_0(D)$.	
	\end{remark}
	
	\noindent
	We now proceed with the improvement of \eqref{logMTsingularinequality}.
	\begin{proposition} \label{MTsingballs}
		Let $N \geq 2$ and  $s\in [0,N)$. For any $u\in W^{1,N}_0(B_1)$ we have 
		\begin{align*}
			\ln \left(\frac{1}{|B_1|^{1-\frac{s}{N}}} \int_{B_1}  \frac{e^u}{|x|^s} \, dx \right)
			< \frac{1}{\widetilde{\omega_{N,s}}}\int_{B_1}| \nabla u|^N \, dx 
			+ \sum_{k=1}^{N-1}\frac{1}{k} + \ln\left(\frac{N}{N-s} \right)+ \frac{s}{N} \ln V_N .
		\end{align*} 
		
	\end{proposition}
	
	\begin{proof}
		
		First, we notice that $D:=\{x\in \mathbb{R}^N: |x|^{\frac{s}{N-s}}x\in B_1 \} = B_1$.\\
		By applying \eqref{Prop3} of Lemma \ref{LamLuLemma} with $ \displaystyle f(x)=e^{\left(\frac{N}{N-s}\right)^{\frac{1}{N'}} x}$, we have
		\begin{align*}
			\displaystyle 
			\int_{B_1} \frac{e^{u(y)}}{|y|^s} \, dy 
			& = \int_{B_1} \frac{f \left(\left(\frac{N-s}{N}\right)^{1/N'}u(y)\right)}{|y|^s} \, dy \\
			& =  \frac{N}{N-s} \int_{B_{1}} f(D_{N,s}u (x)) \, dx\\& = \frac{N}{N-s} \int_{B_{1}} e^{\left(\frac{N}{N-s}\right)^{\frac{1}{N'}}  D_{N,s}u (x)} \, dx.
		\end{align*} 
		
		\noindent
		Therefore, we obtain
		\begin{align*}
			&\ln\left( \frac{1}{|B_1|^{1-\frac{s}{N}}}\int_{B_1} \frac{ e^{u(y)} }{|y|^s}\, dy\right)\\	&=\color{black}\ln\left(\frac{N}{N-s}\frac{1}{|B_1|^{1-\frac{s}{N}}}\int_{B_{1}} e^{(\frac{N}{N-s})^{\frac{1}{N'}} D_{N,s}u(x) }\, dx \right) \\		&=\color{black}\ln\left(\frac{N}{N-s} \right)  + \frac{s}{N} \ln V_N  +\ln\left(\frac{1}{| B_{1}|}\int_{B_{1}} e^{(\frac{N}{N-s})^{\frac{1}{N'}} D_{N,s}u(x) }\, dx \right).
		\end{align*}
		
		\noindent
		By Remark \ref{DNuSob} we know that $D_{N,s}(u)\in W^{1,N}_0(B_{1})$, so we can apply 
		\eqref{InequalityCC} and property \eqref{Prop4} of Lemma \ref{LamLuLemma} to deduce that
		\begin{align*}
			&\ln\left( \frac{1}{|B_1|^{1-\frac{s}{N}}}\int_{B_1}  \frac{e^{u(y)}}{|y|^s} \, dy\right)\\ 
			&<\color{black}\ln\left(\frac{N}{N-s} \right)+ \frac{s}{N} \ln V_N +\left(\frac{N}{N-s}\right)^{N-1}\frac{1}{\widetilde{\omega_{N}}}\int_{ B_{1}}|\nabla D_{N,s}(u)|^N \, dx + \sum_{k=1}^{N-1}\frac{1}{k}  
			\\ &\leq\color{black}  \frac{1}{\widetilde{\omega_{N,s}}}\int_{B_1}| \nabla u|^N \, dy  + \sum_{k=1}^{N-1}\frac{1}{k} + \ln\left(\frac{N}{N-s} \right)+ \frac{s}{N} \ln V_N ,
		\end{align*}
		where we have used that $\left(\frac{N}{N-s}\right)^{N-1}\frac{1}{\widetilde{\omega_{N}}}=\frac{1}{\widetilde{\omega_{N,s}}}$.
	\end{proof}
	
	\noindent
	Let us now consider the associated functional $J_s:  W^{1,N}_0(B_1) \to \mathbb{R}$ defined as
	\[J_s(u)=\frac{1}{ \widetilde{\omega_{N,s}}}\int_{B_1}|\nabla u|^N \, dx -\ln \left(\frac{1}{|B_1|^{1 - \frac{s}{N}}} \int_{B_1} e^u \frac{1}{|x|^s} \, dx \right).\]
	
	\noindent
	If $s=0$, that is in the  non-singular case, we recall (see \cite{BorCinMan-2025} and \cite{CarCha-1986}) that the functional  
	\[J(u):=J_0(u)=\frac{1}{\widetilde{\omega_{N}}}\int_{B_1}|\nabla u|^N \, dx -\ln \left(\frac{1}{|B_1|} \int_{B_1} e^u   \, dx \right)\]
	
	\noindent		
	is such that
	\begin{equation}\label{infNOsing}
		\displaystyle \inf_{W^{1,N}_0(B_1)} J= -\sum_{k=1}^{N-1}\frac{1}{k}.
	\end{equation}
	
	\noindent
	Taking into account this, we are able to prove the following proposition.
	
	\begin{proposition}\label{PropInfJs} 
		Let $N \geq 2$ and $s\in [0,N)$. Then
		\[\inf_{W^{1,N}_0(B_1)}{J_s}= - \sum_{k=1}^{N-1}\frac{1}{k} -\ln\left(\frac{N}{N-s}\right) - \frac{s}{N} \ln V_N .\]
		
		\noindent
		
		Moreover, the infimum is not attained.			
	\end{proposition}

	\begin{proof}
		
		First of all, let us show that 
		\begin{equation}\label{infequal}
			\inf_{W^{1,N}_0(B_1)}{J_s}= \inf_{C^*}{J_s},
		\end{equation}
		
		\noindent
		where
		\[C^*=\{v \in W^{1,N}_0(B_1): v \geq 0\text{ is radially symmetric decreasing}\}.\]
		
		\noindent
		In fact, given $u \in W^{1,N}_0(B_1)$ and denoting by $u^*$ the symmetric decreasing rearrangement of $u$, we recall (see \cite{Kes-2006} for details) that 
		
		\begin{itemize}
			\item[•] $ \lVert u^* \rVert_p =  \lVert u \rVert_p $ for any $1 \leq p < \infty$;
			
			\item[•] $ \left( e^{|u(x)|} \right)^* = e^{u^*(x)} \quad \text{for a.e. } x \in B_1;$
			
			\item[•] $ u^* \in W_0^{1,N}(B_1)$ and Poly\'a – Szeg\"o inequality holds: 
			$$ \displaystyle \int_{B_1}|\nabla u^*|^N \, dx  \leq \int_{B_1} |\nabla u|^N \, dx;  $$
			
			\item[•] by Hardy–Littlewood inequality it follows that 
			$$ \displaystyle \int_{B_1} e^{u(x)} \frac{1}{|x|^s} \, dx \leq \int_{B_1} \left( e^{|u(x)|} \right)^* \left( \frac{1}{|x|^s} \right)^* \, dx \leq \int_{B_1} e^{u^*(x)} \frac{1}{|x|^s} \, dx.$$ 
			
		\end{itemize}
		
		\noindent
		As a consequence, $J_s(u^*) \leq  J_s(u)$  for every $u\in W^{1,N}_0(B_1)$, whence \eqref{infequal} holds.
		
		\noindent
		Now, let us show the connection between $J_s$ and $J$ whenever $u\in C^*$.\\
		By applying  Remark  \ref{DNuSob} we infer that $D_{N,s}(u)\in W_0^{1,N}(B_1).$
		Taking also into account that  $D_{N,s}(u)$ is radially symmetric  whenever $u$ is radially symmetric, we deduce that $D_{N,s}(u) \in C^*$ whenever $u \in C^*$. Moreover, from Lemma \ref{LamLuLemma} and the fact that
		$$ \displaystyle \widetilde{\omega_{N,s}}=\left(\frac{N-s}{N}\right)^{N-1} \widetilde{\omega_{N}} =  \left[ \left(\frac{N-s}{N}\right)^{\frac{1}{N'}} \right]^N \widetilde{\omega_{N}},    $$
		
		\noindent
		we have 
		\begin{align*}
			\displaystyle 
			J_s(u)
			& = \frac{1}{\widetilde{\omega_{N,s}}}\int_{B_1}|\nabla u|^N \, dx -\ln \left(\frac{1}{|B_1|^{1-\frac{s}{N}}} \int_{B_1} e^u \frac{1}{|x|^s} \, dx \right) \\			
			& = \frac{1}{\widetilde{\omega_{N}}}\int_{B_1}\left|\nabla \left( \, \left(\frac{N}{N-s}\right)^{\frac{1}{N'}} D_{N,s}(u) \,  \right) \, \right|^N \, dx -\ln\left(\frac{N}{N-s}\right)  \nonumber \\
			& \quad  -\ln \left(\frac{1}{|B_1|} \int_{B_1} e^{\left(\frac{N}{N-s}\right)^{\frac{1}{N'}}D_{N,s}(u)} \,  dx \right)- \frac{s}{N} \ln  V_N  \nonumber\\
			&= J\left( \left(\frac{N}{N-s}\right)^{\frac{1}{N'}} D_{N,s}(u) \right)-\ln\left(\frac{N}{N-s}\right) - \frac{s}{N} \ln  V_N.\nonumber
		\end{align*}
		
		\noindent
		Thus
		\begin{equation}\label{inf1}
			\displaystyle \inf_{C^*} J_s \geq \inf_{C^*} J -\ln\left(\frac{N}{N-s}\right) - \frac{s}{N} \ln  V_N.
		\end{equation}
		
		\noindent
		Similarly, by using the map $D_{N,s}^{-1}$,  for any $v \in C^*$ we have 
		$$ \displaystyle J(v)= J_s\left( \left(\frac{N-s}{N}\right)^{\frac{1}{N'}} D_{N,s}^{-1}(v) \right) + \ln\left(\frac{N}{N-s}\right) + \frac{s}{N} \ln V_N ,$$
		
		\noindent
		hence
		\begin{equation}\label{inf2}
			\displaystyle \inf_{C^*} J_s \leq \inf_{C^*} J - \ln\left(\frac{N}{N-s}\right) - \frac{s}{N} \ln  V_N.
		\end{equation}
		
		\noindent
		Therefore, from \eqref{infNOsing}, \eqref{inf1} and \eqref{inf2}, we have
		\begin{align*}
			\displaystyle 
			\inf_{W_0^{1,N}(B_1)} J_s 
			& = \inf_{W_0^{1,N}(B_1)} J - \ln\left(\frac{N}{N-s}\right) - \frac{s}{N} \ln  V_N \\
			& = - {\sum_{k=1}^{N-1}} \frac{1}{k} - \ln\left(\frac{N}{N-s}\right) - \frac{s}{N} \ln  V_N.
		\end{align*} 
		
		\noindent
		Finally, by Proposition \ref{MTsingballs} we deduce that the infimum is not attained. 
		
	\end{proof}
	
	\begin{remark}\label{invariance}
		
		Denoting by $B_r$ the ball of $\mathbb{R}^N$ centered in the origin and with radius $r>0$,	by rescaling we get
		\begin{align*}
			\displaystyle 
			\inf_{ u \in W_0^{1,N}(B_r)} \left\lbrace \frac{1}{ \widetilde{\omega_{N,s}}}\int_{B_r}|\nabla u|^N \, dx -\ln \left(\frac{1}{|B_r|^{1 - \frac{s}{N}}} \int_{B_r}  \frac{ e^u}{|x|^s} \, dx \right) \right\rbrace 
			& = \inf_{ u \in W_0^{1,N}(B_1)} J_s(u). \\
		\end{align*}
	\end{remark}

	\noindent
	Now, we show the equivalence result. Since the reasoning is analogous to \cite{BorCinMan-2025}, we briefly outline the key steps of the proof.
	
	\begin{proof}[\textbf{Proof of Theorem \ref{Equivalence}}]
		\noindent First, the inequality $``\geq"$ is deduced.
		\par 
		\noindent			
		Let $u\in W_{\mu_{N,s}}(\mathbb{R}^N)$. By Theorem \ref{densitystatementANDextension}, we can assume $u\in C^{\infty}_0(\mathbb{R}^N)$ and extend the result by density. We consider $\eta \in C^{\infty}_0(\mathbb{R}^N)$ such that $0 \leq \eta \leq 1$, $ \eta \equiv 1$ if $ \displaystyle |x| \leq \frac{1}{2},$ and $\eta \equiv 0 $ if $\displaystyle  |x| \geq 1$. Let us define, for $r>0$, $\displaystyle \eta_r(x):=\eta \left(\frac{x}{r}\right)$. Obviously, $\displaystyle \eta_r(x) \to 1$ as $r \to +\infty,$ for any $ x \in \mathbb{R}^N.$\\
		We will denote by $v_{N,s}(x):= \ln \mu_{N,s}(x)$ and by $o(1)$ a quantity that goes to zero as $r \to + \infty.$
		
		\par 
		\noindent Through the Lebesgue Convergence Theorem we deduce that 
		\begin{equation*}
			\displaystyle \ln\left(\int_{B_r}e^{u(x)\eta_r+v_{N,s}(x)}\,  dx \right)\rightarrow \ln \left(\int_{\mathbb{R}^N}e^{u(x)} \, d\mu_{N,s}\right).
		\end{equation*}
		
		\noindent
		Now, let us consider the function $w:=u\eta_r+g_{N,s}(x)-g_{N,s}(r)$,	where $g_{N,s}(r)$ denotes the function $g_{N,s}(x)$ evaluated at $|x|=r$. Since $g_{N,s}(x) \in W^{1,N}_{loc}(\mathbb{R}^N)$, it follows that $w\in W^{1,N}_0(B_r)$. Let us make the following computations:
		\begin{align*}
			\ln & \left(\int_{B_r}e^{u(x)\eta_r+v_{N,s}(x)} \, dx \right)  =  \ln\left(e^{g_{N,s}(r)}\int_{B_r} \frac{e^{w(x)} }{|x|^s} \, dx\right)  \\
			& =   \ln\left(\frac{N-s}{\omega_{N-1}}\frac{|B_r|^{1-\frac{s}{N}}}{(1+r^{\frac{N-s}{N-1}})^N} \frac{1}{|B_r|^{1-\frac{s}{N}}}\int_{B_r} \frac{e^{w(x)} }{|x|^s}\, dx\right) \nonumber	 \\
			&=	\ln\left(\frac{(N-s)V_N^{-\frac{s}{N}}}{N} \right) + \ln\left(\frac{1}{|B_r|^{1-\frac{s}{N}}} \int_{B_r}  \frac{e^{w(x)} }{|x|^s}\,  dx\right) + \ln \left( \frac{r^{N-s}}{(1+r^{\frac{N-s}{N-1}})^N} \right)	\nonumber \\
			&= \ln\left(\frac{1}{|B_r|^{1-\frac{s}{N}}} \int_{B_r}  \frac{e^{w(x)} }{|x|^s}\,  dx\right) - \ln \left( \frac{N}{N-s} \right) - \frac{s}{N} \ln V_N 
			- \frac{N-s}{N-1} \ln(r)+o(1). \nonumber
		\end{align*}	
		Since $w\in W^{1,N}_0(B_r)$, by Remark \ref{invariance} and Proposition \ref{PropInfJs}  we infer that 
		\begin{align*}
			\displaystyle 
			\ln\left(\frac{1}{|B_r|^{1-\frac{s}{N}}}\int_{B_r} \frac{e^{w(x)} }{|x|^s} \, dx \right)
			& < \frac{1}{\widetilde{\omega_{N,s}}}\int_{B_r}|\nabla w|^N \, dx -\inf_{W^{1,N}_0(B_1)} J_s.
		\end{align*}			
		
		\noindent
		Regarding $ \displaystyle \frac{1}{\widetilde{\omega_{N,s}}}\int_{B_r}|\nabla w|^N \, dx $, due to the definition of $w$ and $H_{N,s}$, we get
		\begin{align*}
			\displaystyle
			\frac{1}{\widetilde{\omega_{N,s}}} \int_{B_r}|\nabla w|^N \, dx =&  \frac{1}{\widetilde{\omega_{N,s}}} \left(\int_{B_r}H_{N,s}(u\eta_r) \, dx + \int_{B_r}|\nabla g_{N,s}(x)|^N \, dx\right.\\ &\left.+ N \int_{B_r}|\nabla g_{N,s}(x)|^{N-2}\nabla g_{N,s}(x) \cdot \nabla(u\eta_r)\, dx\right).
		\end{align*}
		
		\noindent
		Moreover, by Green's formula and Remark \ref{Nlaplaciano},  
		\begin{align*}
			\displaystyle
			\frac{N}{\widetilde{\omega_{N,s}}} \int_{B_r}|\nabla g_{N,s}(x)|^{N-2}\nabla g_{N,s}(x)\cdot \nabla(u\eta_r)\, dx = \int_{B_r} u \eta_r  \, d \mu_{N,s} .
		\end{align*}
		
		\noindent
		By Remark \ref{RemarkIntegral}, we have 
		\begin{equation}\label{calcolointegralevns}
			\displaystyle 
			\frac{1}{\widetilde{\omega_{N,s}}}\int_{B_r}|\nabla g_{N,s}(x) \, |^N \, dx =\frac{N-s}{N-1} \ln r - \sum_{k=1}^{N-1} \frac{1}{k} + o(1).
		\end{equation}
		
		\noindent
		Summing up, we have shown that 
		\begin{align*}
			\displaystyle  
			\ln\left(\int_{B_r}e^{u(x)\eta_r+v_{N,s}(x)} \, dx \right) 	< & \, \frac{1}{\widetilde{\omega_{N,s}}} \int_{B_r}H_{N,s}(u\eta_r) \, dx  + \int_{B_r} u \eta_r  \, d \mu_{N,s}\nonumber \\
			& \, - \inf_{W^{1,N}_0(B_1)} J_s - \ln \left( \frac{N}{N-s} \right) - \frac{s}{N} \ln V_N - \sum_{k=1}^{N-1} \frac{1}{k} + o(1). \nonumber
		\end{align*}
		
		\noindent
		Finally, since $\eta_r  \equiv 0$ for $|x|>r$, the definition of $I_s$ gives
		\begin{equation}\label{contoprimadellimite}
			\displaystyle \inf_{W^{1,N}_0(B_1)} J_s +\sum_{k=1}^{N-1}\frac{1}{k} +\ln\left(\frac{N}{N-s}\right) + \frac{s}{N} \ln V_N <  I_s(u\eta_r)+o(1).
		\end{equation}
		
		\noindent
		Now, we pass to the limit as $r \to + \infty$.  Since $u \in C_0^{\infty}(\mathbb{R}^N)$ and $\eta_r(x) \to 1$ for almost every $x \in \mathbb{R}^N$, by the Dominated Convergence Theorem we get 
		$$ \dis \ln\left(\int_{B_r}e^{u(x)\eta_r+v_{N,s}(x)} \, dx \right) \to \ln\left(\intrn e^{u(x)} \ d \mu_{N,s}  \right) \qquad \text{and} \qquad \int_{B_r} u \eta_r  \, d \mu_{N,s} \to \intrn u  \, d \mu_{N,s}.$$
		Moreover, since $ \displaystyle H_{N,s}(u(x) \eta_r(x))  \to H_{N,s}(u(x))$ for almost every $x \in \mathbb{R}^N$, and by Lemma \ref{Lemma 2.1} and Remark \ref{proprietagradvnslnxs}, we infer the existence of a function $v \in  L^1(\mathbb{R}^N, dx)$ such that, for any $ x \in \mathbb{R}^N$ and for any $r>1$, we have $ H_{N,s}(u(x) \eta_r(x)) \leq v(x)$, we can use again the Dominated Convergence Theorem  to have that
		$$ \dis  \frac{1}{\widetilde{\omega_{N,s}}} \int_{B_r}H_{N,s}(u\eta_r) \, dx \to  \frac{1}{\widetilde{\omega_{N,s}}} \intrn H_{N,s}(u) \, dx.$$
		Then, we conclude that 
		$$ \displaystyle I_s(u\eta_r) \to I_s(u) \qquad \text{as } r \to \infty,$$
		which, together with \eqref{contoprimadellimite}, yields
		\[	\inf_{W^{1,N}_0(B_1)} J_s + \sum_{k=1}^{N-1}\frac{1}{k} +\ln\left(\frac{N}{N-s}\right) + \frac{s}{N} \ln V_N \leq I_s(u). \]
		
		\noindent
		By Theorem \ref{densitystatementANDextension}, the above inequality  holds for any $u \in W_{\mu_{N,s}}(\mathbb{R}^N)$, hence
		$$
		\inf_{W^{1,N}_0(B_1)}  J_s + \sum_{k=1}^{N-1}\frac{1}{k} +\ln\left(\frac{N}{N-s}\right) + \frac{s}{N} \ln V_N  \leq \inf_{W_{\mu_{N,s}}(\mathbb{R}^N)} I_s.
		$$
		
		\par 
		\noindent
		Finally, we obtain the inequality $``\leq"$.\par 
		\noindent
		Let $u\in W^{1,N}_0(B_1)$, $r>0$ and let us define the function $u_r: \mathbb{R}^N \to \mathbb{R}$ as 
		\begin{equation*}
			u_r(x)=\label{problemaoriginal}
			\begin{cases}
				\displaystyle u\left(\frac{x}{r}\right)-g_{N,s}(x)+ g_{N,s}(r), &  \text{if $|x|\leq r$,}  \\
				0 & \text{if $|x|>r$.}
			\end{cases}
		\end{equation*}
		
		\noindent
		Note that $u_r \in W^{1,N}(\mathbb{R}^N)$. From Lemma \ref{Sobole-NewSpace}, $u_r \in W_{\mu_{N,s}}(\mathbb{R}^N)$, so
		\begin{equation}\label{Step2_1}
			\displaystyle					\inf_{W_{\mu_{N,s}}(\mathbb{R}^N)} I_s\leq I_s(u_r),
		\end{equation}
		\noindent
		where
		\begin{equation*}
			\displaystyle I_s(u_r) = \frac{1}{\widetilde{\omega_{N,s}}} \int_{\mathbb{R}^N} H_{N,s}(u_r) \, dx + \int_{\mathbb{R}^N} u_r \, d\mu_{N,s} - \ln \left( \int_{\mathbb{R}^N} e^{u_r} \, d\mu_{N,s} \right).
		\end{equation*}
		
		\noindent
		Let us compute each integral of the functional.
		\begin{align*}
			\displaystyle
			\ln\left(\int_{\mathbb{R}^N} e^{u_r} \, d\mu_{N,s}\right)=&\ln\left( r^s\mu_{N,s}(r)\int_{B_r} \frac{e^{u\left(\frac{x}{r}\right)}}{|x|^s} \, dx+\int_{\mathbb{R}^N\setminus B_r}\mu_{N,s}(x) \, dx \right)\\
			= & \ln \left(  \frac{1}{|B_1|^{1-\frac{s}{N}}} \int_{B_1}  \frac{e^{u}}{|x|^s} \, dx + \frac{N(N-1)V_N^{\frac{s}{N}}}{N-s} + o(1)  \right) \\
			& - \frac{N-s}{N-1} \ln r - \ln \left( \frac{N}{N-s} \right) - \frac{s}{N} \ln V_N + o(1),
		\end{align*}
		where in the last equality we used that
		\begin{align*}
			&\displaystyle
			r^s\mu_{N,s}(r) \int_{B_r} \frac{e^{u\left(\frac{x}{r}\right)}}{|x|^s} \, dx 
			= \frac{N-s}{NV_N^{\frac{s}{N}}} \frac{r^{N-s}}{\left( 1+r^{\frac{N-s}{N-1}}  \right)^N}\frac{1}{|B_1|^{1-\frac{s}{N}}} \int_{B_1}  \frac{e^{u(y)}}{|y|^s} \, dy,
		\end{align*}
		and
		\begin{align*}
			\displaystyle 
			\int_{\mathbb{R}^N \setminus B_r}\mu_{N,s}(x) \, dx 
			=  1 - \frac{r^{N-s}}{\left(1+r^{\frac{N-s}{N-1}}\right)^{N-1}}.
		\end{align*}
		
		\noindent
		Now, we compute  $ \displaystyle \frac{1}{\widetilde{\omega_{N,s}}}\int_{\mathbb{R}^N} H_{N,s}(u_r) \, dx$.
		\begin{align*}
			\displaystyle 
			\frac{1}{\widetilde{\omega_{N,s}}} \int_{\mathbb{R}^N} H_{N,s}(u_r) \, dx
			= & \frac{1}{\widetilde{\omega_{N,s}}} \int_{B_1} |\nabla u |^N \, dx  + \frac{N-1}{\widetilde{\omega_{N,s}}} \int_{B_r} |\nabla g_{N,s}|^N \, dx \\
			& - \frac{N}{\widetilde{\omega_{N,s}}} \int_{B_r} |\nabla g_{N,s}(x)|^{N-2} \nabla  g_{N,s}(x) \cdot \nabla_x u \left( \frac{x}{r} \right) \, dx\\
			= & \frac{1}{\widetilde{\omega_{N,s}}} \int_{B_1} |\nabla u |^N \, dx 
			- \int_{B_r} u\left(\frac{x}{r} \right) \, d \mu_{N,s}   
			\\&+ (N-s) \ln r - (N-1) \sum_{k=1}^{N-1} \frac{1}{k} + o(1).
		\end{align*}
		\noindent
		In the last inequality we have taken into account \eqref{calcolointegralevns}, along with the Green's formula and Remark \ref{Nlaplaciano}.\par 
		\noindent
		On the other hand, since 
		\begin{align*}
			\displaystyle
			- N \ln \left( 1 + r^{\frac{N-s}{N-1}}   \right) \int_{B_r}  \,  d\mu_{N,s} 
			& = - N \ln \left( 1 + r^{\frac{N-s}{N-1}}   \right) \frac{r^{N-s}}{ ( 1 + r^{\frac{N-s}{N-1}} )^{N-1}} \\
			& =  - \frac{N(N-s)}{N-1} \ln r + o(1)
		\end{align*}
		and \eqref{calcolointegralevns} is satisfied, we have
		$$ \int_{\mathbb{R}^N} u_r \, d\mu_{N,s} = \int_{B_r} u \left( \frac{x}{r} \right) \, d\mu_{N,s} + N \sum_{k=1}^{N-1} \frac{1}{k}  - \frac{N(N-s)}{N-1} \ln r + o(1).$$
		
		\noindent
		Summing up and taking into account \eqref{Step2_1}, we obtain 
		\begin{align*}
			\displaystyle
			\inf_{W_{\mu_{N,s}}(\mathbb{R}^N)} I_s
			\leq&\frac{1}{\widetilde{\omega_{N,s}}}\int_{B_1} |\nabla u|^N  \, dx  - \ln \left(  \frac{1}{|B_1|^{1-\frac{s}{N}}} \int_{B_1} \frac{e^{u}}{|x|^s} \, dx  + \frac{N(N-1)V_N^{\frac{s}{N}}}{N-s} + o(1) \right)\\ 
			& + \ln \left( \frac{N}{N-s} \right) + \frac{s}{N} \ln V_N + \sum_{k=1}^{N-1} \frac{1}{k} + o(1).
		\end{align*}
		Therefore, passing to the limit as $r \to \infty$, 
		\begin{align*}
			\displaystyle
			\inf_{W_{\mu_{N,s}}(\mathbb{R}^N)} I_s
			\leq&\frac{1}{\widetilde{\omega_{N,s}}}\int_{B_1} |\nabla u|^N  \, dx  - \ln \left(  \frac{1}{|B_1|^{1-\frac{s}{N}}} \int_{B_1}  \frac{e^{u}}{|x|^s} \, dx  + \frac{N(N-1)V_N^{\frac{s}{N}}}{N-s}  \right)\\ 
			& \qquad + \ln \left( \frac{N}{N-s} \right) + \frac{s}{N} \ln V_N + \sum_{k=1}^{N-1} \frac{1}{k} \\
			<&  J_s(u) + \ln \left( \frac{N}{N-s} \right) + \frac{s}{N} \ln V_N + \sum_{k=1}^{N-1} \frac{1}{k}.
		\end{align*}
		
		\noindent
		Since the inequality holds for any $u\in W_0^{1,N}(B_1)$, we get
		$$ \displaystyle \inf_{W_{\mu_{N,s}}(\mathbb{R}^N)} I_s \leq  \inf_{W_0^{1,N}(B_1)} J_s + \ln \left( \frac{N}{N-s} \right) + \frac{s}{N} \ln V_N + \sum_{k=1}^{N-1} \frac{1}{k}.$$
	\end{proof}
	
	\begin{remark}
		We emphasize that the proof of the previous theorem actually yields a stronger inequality than \eqref{Chang-Carleson-Singularity}. More precisely, for any $s \in [0,N)$ and any $u \in W^{1,N}_0(B_1)$, we have
		\begin{align*}
			\displaystyle
			\ln &\left(  \frac{1}{|B_1|^{1-\frac{s}{N}}} \int_{B_1}  \frac{e^{u}}{|x|^s} \, dx  + \frac{N(N-1)V_N^{\frac{s}{N}}}{N-s}  \right)\\
			&\leq\frac{1}{\widetilde{\omega_{N,s}}}\int_{B_1} |\nabla u|^N  \, dx  + \sum_{k=1}^{N-1} \frac{1}{k} + \ln \left( \frac{N}{N-s} \right) + \frac{s}{N} \ln V_N . 
		\end{align*}
		This inequality is particularly interesting for $s=0$, since it sharpens the classical Carleson-Chang inequality \eqref{InequalityCC}. 
		
	\end{remark}
	
	\appendix
	\section{Regularity of the functional}\label{app}
	
	In this section we prove the Fréchet differentiability of the functional $I_{s,\lambda}$ in $\wmns$.

	\begin{theorem}\label{Diff}
		The functional $I_{s,\lambda}$ is of class $C^1$ for any  $\lambda>0$. In particular, for any $u, v \in  W_{\mu_{N,s}}(\mathbb{R}^N)$ we have
		\begin{align*}
			\langle I'_{s,\lambda} (u), v \rangle=&\frac{N}{\widetilde{\omega_{N,s}}}\int_{\mathbb{R}^N} \left[|\nabla g_{N,s}+\nabla u|^{N-2}(\nabla g_{N,s}+\nabla u)  - |\nabla g_{N,s}|^{N-2}\nabla g_{N,s} \right] \cdot \nabla v \, dx \\
			& + \lambda\int_{\mathbb{R}^N} v \, d\mu_{N,s}  -\lambda\frac{\dis \int_{\mathbb{R}^N}e^u v \, d\mu_{N,s}}{ \dis \int_{\mathbb{R}^N}e^u \, d\mu_{N,s}}.
		\end{align*}
	\end{theorem}
	
	We start by stating some elementary lemmas.
	
	\begin{lemma}\label{EstGatoux} For $N\ge 3$, and for any $X,Y \in \R^N$, we have that 
		\begin{equation}\label{EstXY}
			||X+Y|^{N-2} - |X|^{N-2}|\le (N-2) (|X|+ |Y|)^{N-3} |Y|.
		\end{equation}
		In particular, there exists $C_N$ depending only on $N$ such that 
		\begin{equation}\label{EstXY2}
			\left||X+Y|^{N-2} - |X|^{N-2}\right|\le C_N \left( |X|^{N-3}|Y| + |Y|^{N-2}\right).
		\end{equation}
		and  
		\begin{equation}\label{EstXY3}
			\left||X+Y|^{N-2}(X+Y) - |X|^{N-2}X\right|\le C_N \left(|X|^{N-2} +|Y|^{N-2}\right)|Y|.
		\end{equation}
		Moreover, for any $X,Y,Z \in \R^N$ we have
		\begin{equation}\label{EstXYZ}
			\left||X+Y+Z|^{N-2} - |X|^{N-2}\right|\le C_N \left( |X|^{N-3}|Y| + |X|^{N-3}|Z| + |Y|^{N-2}+ |Z|^{N-2}\right).
		\end{equation}

	\end{lemma}
	\begin{proof}
		When $N=3$, the three inequalities trivially follow from the Lipschitzianity of the Euclidean norm in $\R^N$. If $N\ge 4$, \eqref{EstXY} is obtained by applying the mean value theorem to the $C^1$ function $H_{X}(Y)= |X+Y|^{N-2}$, observing that if $\xi_Y \in [0,Y]$ then
		$$
		|\nabla H_X(\xi_Y)|\le (N-2) (|X|+|Y|)^{N-3}. 
		$$ 
		To prove \eqref{EstXY2}, we just use the convexity inequality 
		$$
		(|X|+|Y|)^{N-3}\le 2^{N-4} (|X|^{N-3}+|Y|^{N-3}).
		$$
		To prove \eqref{EstXY3}, we observe that 
		$$
		\begin{aligned}
			\left||X+Y|^{N-2}(X+Y) - |X|^{N-2}X\right| \le \left||X+Y|^{N-2} -|X|^{N-2}\right||X|+ |X+Y|^{N-2}|Y|.
		\end{aligned}
		$$
		Hence, using \eqref{EstXY} to estimate the first term, we get
		$$
		\begin{aligned}
			\left||X+Y|^{N-2}(X+Y) - |X|^{N-2}X\right| &\le (N-2)(|X|+|Y|)^{N-3}|X||Y| + (|X| +|Y|)^{N-2}|Y|\\ & \le (N-1)(|X|+|Y|)^{N-2}|Y| \\
			& \le (N-1)2^{N-3}(|X|^{N-2}+|Y|^{N-2})|Y|.
		\end{aligned}
		$$
		Finally \eqref{EstXYZ} is obtained by \eqref{EstXY2} replacing $Y$ with $Y+Z$ and estimating 
		$$
		(|Y|+|Z|)^{N-2}\le 2^{N-3} (|Y|^{N-2}+|Z|^{N-2}).
		$$
	\end{proof}

	\begin{lemma}\label{LemmaAppIntegrability}
		Let us assume $N >2$ and $s \in [0,N)$. Then the following estimates are satisfied:
		\begin{enumerate}
			\item	For any $u,v\in \wmns$,we have $	|\nabla g_{N,s}||\nabla u|^{N-2}|\nabla v| \in L^1(\R^N,dx)$ and 
			$$
			\int_{\R^N} |\nabla g_{N,s}||\nabla u|^{N-2}|\nabla v|\, dx \le  \|u\|_{\mu_{N,s}}^{N-2} \|v\|_{\mu_{N,s}}.
			$$ 
			\item	For any $v\in \wmns$,we have $	|\nabla g_{N,s}||\nabla v|^{N-1} \in L^1(\R^N,dx)$ and 
			$$
			\int_{\R^N} |\nabla g_{N,s}||\nabla v|^{N-1}\, dx \le  \|v\|_{\mu_{N,s}}^{N-1}.
			$$ 
		\end{enumerate}
	\end{lemma}
	\begin{proof}
		The first statement follows by H\"older inequality with exponents $2$ and $2$ if $N =3$. Otherwise, it follows from the extended H\"older inequality with exponents $N$, $\frac{N(N-2)}{2}$, $\frac{N^2}{N-2}$ and $\frac{N}{N-2-\frac{4}{N(N-2)}}$, since
		$$
		\begin{aligned}
			& \left| \nabla g_{N,s} \right|  \left| \nabla v   \right| \left| \nabla u   \right|^{N-2}\\
			&= \left| \nabla g_{N,s} \right|^{\frac{N-2}{N}} \left| \nabla v \right|^{\frac{2}{N}} \left| \nabla g_{N,s} \right|^{\frac{2}{N}} \left| \nabla u \right|^{\frac{4}{N(N-2)}} \left| \nabla v \right|^{\frac{N-2}{N}} \left| \nabla u \right|^{N-2-\frac{4}{N(N-2)}}
		\end{aligned}
		$$
		and
		\begin{equation}
			\begin{aligned}
				\dis 
				\int_{\mathbb{R}^N}   & \left| \nabla g_{N,s} \right| \left| \nabla v \right|  \left| \nabla u   \right|^{N-2} \, dx  \\
				&\leq \left( \int_{\mathbb{R}^N}  |\nabla g_{N,s}|^{N-2}   \left| \nabla v \right|^2 \, dx       \right)^{\frac{1}{N}} \left( \int_{\mathbb{R}^N}  |\nabla g_{N,s}|^{N-2}   \left| \nabla u \right|^2 \, dx       \right)^{\frac{2}{N(N-2)}} \\&\quad \times \lVert \nabla v \rVert_N^{\frac{N-2}{N}}  \lVert \nabla  u \rVert_N^{N-2-\frac{4}{N(N-2)}}\le   \|u\|_{\mu_{N,s}}^{N-2} \|v\|_{\mu_{N,s}}.
			\end{aligned}
		\end{equation}
		For the second statement, it is sufficient to take $u=v$.
	\end{proof}

	We now present the proof of Theorem \ref{Diff}, which will be divided into two propositions.

	\begin{proposition}\label{DiffHNs}
		The functional $P_{N,s} : W_{\mu_{N,s}}(\mathbb{R}^N) \to \mathbb{R}$ defined as 
		$$ \displaystyle P_{N,s}(u):=  \int_{\mathbb{R}^N} H_{N,s}(u) \, dx$$
		
		\noindent
		is of class $C^1$ in $\wmns$, and  for any $u, v \in  W_{\mu_{N,s}}(\mathbb{R}^N)$ we have
		\begin{align*}
			\langle P'_{N,s}(u), v \rangle =N\int_{\mathbb{R}^N} \left[ \, |\nabla g_{N,s}+\nabla u|^{N-2}(\nabla g_{N,s}+\nabla u)  - |\nabla g_{N,s}|^{N-2}\nabla g_{N,s} \right] \cdot \nabla v \, dx. 
		\end{align*}
	\end{proposition}
	\begin{proof}
		We provide the proof for $N \geq 3$, since it is standard for  $N=2$.
		
		\noindent
		We first prove that $P_{N,s}$ is Gâteaux differentiable, that is for any $u,v \in  W_{\mu_{N,s}}(\mathbb{R}^N)$ we have
		\begin{equation}\label{DirDer}
			\begin{aligned}
				&\langle P'_{N,s}(u), v\rangle= 
				\left.\frac{d}{dt}
				P_{N,s}(u + tv) \right|_{t=0}\\
				& = N\int_{\mathbb{R}^N} \left[ \, |\nabla g_{N,s}+\nabla u|^{N-2}(\nabla g_{N,s}+\nabla u)  - |\nabla g_{N,s}|^{N-2}\nabla g_{N,s} \right] \cdot \nabla v \, dx
			\end{aligned} 
		\end{equation}
		and that $v\rightarrow \langle P'_{N,s}(u), v\rangle$ is linear and continuous.
		
		\noindent
		Since  $H_{N,s}(u)\in L^1(\R^N,dx)$ and the function $t \mapsto H_{N,s}(u+tv)$ is of class $C^1$ with 
		$$
		\frac{d}{dt} H_{N,s}(u+tv) =
		N \left[|\nabla g_{N,s}+\nabla u+t \nabla v|^{N-2}(\nabla g_{N,s}+\nabla u+t \nabla v)  - |\nabla g_{N,s}|^{N-2}\nabla g_{N,s} \right] \cdot \nabla v,
		$$ by differentiation under integral sign, it is sufficient to show that there exists $ \psi \in L^1(\R^N,dx)$ such that for any $t \in [-1,1]$:
		$$
		N \left| |\nabla g_{N,s}+\nabla u+t \nabla v|^{N-2}(\nabla g_{N,s}+\nabla u+t \nabla v)  - |\nabla g_{N,s}|^{N-2}\nabla g_{N,s} \right|| \nabla v| \le \psi.
		$$
		We first observe that
		\begin{equation}\label{contoLagrangePiuVariabili}
			\begin{aligned}
				&\left| |\nabla g_{N,s} +\nabla u+t \nabla v|^{N-2}(\nabla g_{N,s}+\nabla u+t \nabla v)  - |\nabla g_{N,s}|^{N-2}\nabla g_{N,s} \right|| \nabla v|  \\
				& \qquad \le \left| |\nabla g_{N,s}+\nabla u+t \nabla v|^{N-2} - |\nabla g_{N,s}|^{N-2} \right||\nabla g_{N,s} || \nabla v| \\ 
				&\qquad \quad+  |\nabla g_{N,s}+\nabla u+t \nabla v|^{N-2} (|\nabla u| + |\nabla v|)|\nabla v|. 
			\end{aligned}
		\end{equation}
		Applying Lemma \ref{EstGatoux}, specifically, \eqref{EstXYZ} with $X= \nabla g_{N,s}$, $Y = \nabla u$ and $Z= \nabla v$, we find that 
		$$
		\begin{aligned}
			& \left| |\nabla g_{N,s}+\nabla u+t \nabla v|^{N-2} - |\nabla g_{N,s}|^{N-2} \right||\nabla g_{N,s} || \nabla v| \\
			& \qquad \le C_N (|\nabla g_{N,s}|^{N-2}|\nabla u||\nabla v| +|\nabla g_{N,s}|^{N-2}|\nabla v|^2 + |\nabla g_{N,s}||\nabla u|^{N-2} |\nabla v|+  |\nabla g_{N,s}| |\nabla v|^{N-1}).
		\end{aligned}
		$$
		By H\"older inequality and Lemma \ref{LemmaAppIntegrability} all terms in the RHS belong to $L^1(\R^N,dx)$. Similarly, we have that
		$$\begin{aligned}
			&	|\nabla g_{N,s}+\nabla u+t \nabla v|^{N-2} (|\nabla u| + |\nabla v|)|\nabla v|  \\ & \qquad \le C_N  \left(|\nabla g_{N,s}|^{N-2}|\nabla u| |\nabla v| + |\nabla g_{N,s}|^{N-2}|\nabla v|^2 + (|\nabla u|+|\nabla v|)^{N-1}|\nabla v| \right)\\
			& \qquad \le C_N\left( |\nabla g_{N,s}|^{N-2}|\nabla u| |\nabla v| + |\nabla g_{N,s}|^{N-2}|\nabla v|^2 + |\nabla u|^{N-1}|\nabla v|+|\nabla v|^N\right), 
		\end{aligned}
		$$ 
		with all the terms in $RHS$ in $L^1(\R^N,dx)$.

		\noindent
		Finally, to show that $P_{N,s}$ is Gâteaux differentiable on $\wmns$, we need to prove linearity and the continuity of $v\rightarrow \langle P'_{N,s}(u), v\rangle$ for any $u$, $v \in W_{\mu_{N,s}}(\mathbb{R}^N)$.
		The linearity is obvious, so let us focus on the continuity.\\
		
		\noindent
		In order to guarantee continuity, it is enough if the functional is bounded by $\|v\|_{\mu_{N,s}}$. Let us fix $u \in  W_{\mu_{N,s}}(\mathbb{R}^N)$, then, by taking into account  \eqref{contoLagrangePiuVariabili}, we have that for any $v \in  W_{\mu_{N,s}}(\mathbb{R}^N)$,
		\begin{align*}
			\left| \langle P'_{N,s}(u), v \rangle   \right| & = \bigg|N\int_{\mathbb{R}^N} \left[ \, |\nabla g_{N,s}+\nabla u|^{N-2}(\nabla g_{N,s}+\nabla u)  - |\nabla g_{N,s}|^{N-2}\nabla g_{N,s} \right] \cdot \nabla v \, dx \bigg| \\
			&\leq N \int_{\R^N }\bigg|  \,  |\nabla g_{N,s}+\nabla u|^{N-2} - |\nabla g_{N,s}|^{N-2} \bigg|   \left| \nabla g_{N,s} \right|  \left| \nabla v \right|  \\&\quad + C_N \left( |\nabla g_{N,s}|^{N-2}   \left| \nabla u \right|  \left| \nabla v \right| +  \left| \nabla u \right|^{N-1}  \left| \nabla v \right|   \right)\, dx  \\
			&\leq  C_N \int_{\R^N}\Bigg( \left|     \nabla g_{N,s}|^{N-2}   \left| \nabla u \right|  \left| \nabla v \right| + |\nabla g_{N,s} \right| \left| \nabla v \right|  \left| \nabla u   \right|^{N-2}  +  \left| \nabla u \right|^{N-1}  \left| \nabla v \right| \Bigg) dx ,
		\end{align*}
		where in the last inequality we have used Lemma \ref{EstGatoux}.

		\noindent
		By employing H\"older inequality twice, first with exponents $2$ and $2$ and then with $N/(N-1)$ and $N$, we get
		\begin{align*}
			\dis 
			& \int_{\mathbb{R}^N} \left( |\nabla g_{N,s}|^{N-2}   \left| \nabla u \right|  \left| \nabla v \right| +  \left| \nabla u \right|^{N-1}  \left| \nabla v \right|   \right) \, dx \\
			&\leq    \left( \int_{\mathbb{R}^N} |\nabla g_{N,s}|^{N-2}   \left| \nabla u \right|^2 \, dx      \right)^{\frac{1}{2}} \left( \int_{\mathbb{R}^N}  |\nabla g_{N,s}|^{N-2}   \left| \nabla v \right|^2 \, dx     \right)^{\frac{1}{2}} +  \lVert \nabla u \rVert_N^{N-1} \lVert \nabla v \rVert_N \\
			& \le \lVert  u \rVert_{\mu_{N,s}} \lVert  v \rVert_{\mu_{N,s}}+ \lVert  u \rVert_{\mu_{N,s}}^{N-1} \lVert  v \rVert_{\mu_{N,s}}.
		\end{align*}
		Using also Lemma \ref{LemmaAppIntegrability} we obtain that 
		\begin{align*}
			\left| \langle P'_{N,s}(u), v \rangle   \right| &\leq C \left( \|u\|_{\mu_{N,s}} + \|u\|_{\mu_{N,s}}^{N-1}+ \|u\|_{\mu_{N,s}}^{N-2}\right)\|v\|_{\mu_{N,s}}\leq C_{u,g_{N,s}}\|v\|_{\mu_{N,s}}.
		\end{align*}	
In order to show that $P_{N,s}$ is Fr\'echet differentiable, we prove the continuity of  the map  $P'_{N,s}: W_{\mu_{N,s}}(\mathbb{R}^N) \to \left[  W_{\mu_{N,s}}(\mathbb{R}^N) \right] ' $. Let  $u_n \in W_{\mu_{N,s}}(\mathbb{R}^N) $ such that $ u_n \to u$ in $W_{\mu_{N,s}}(\mathbb{R}^N) $, then
		\begin{align*}
			\left|\langle P'_{N,s}(u) -  P'_{N,s}(u_k) , v \rangle \right|
			\leq C_{N,s} \int_{\mathbb{R}^N}& \bigg| |\nabla g_{N,s}+\nabla u|^{N-2}(\nabla g_{N,s}+\nabla u)
			\\&- |\nabla g_{N,s}+\nabla u_k|^{N-2}(\nabla g_{N,s}+\nabla u_k) \bigg|  \left| \nabla v \right| \, dx.
		\end{align*}
		Let us bound this expression depending on just the norm of $u,u_n$ in $W_{\mu_{N,s}}(\mathbb{R}^N)$ for any $\|v\|_{\mu_{N,s}}\leq 1$. Applying \eqref{EstXY3} with  $X=\nabla g_{N,s} + \nabla u$ and $Y= \nabla u_n-\nabla u$, we have
		\begin{equation}\label{ineq1}
			\begin{aligned}
				&\left| |\nabla u_n+\nabla g_{N,s}|^{N-2}\nabla (u_n+g_{N,s})-	|\nabla u+\nabla g_{N,s}|^{N-2}\nabla (u+g_{N,s}) \right||\nabla v|
				\\& \qquad \le C_N |\nabla u_n- \nabla u|\left(|\nabla g_{N,s}+ \nabla u|^{N-2}+ |\nabla u_n - \nabla u|^{N-2}\right) |\nabla v|
				\\& \qquad\leq C_N\left( |\nabla g_{N,s}|^{N-2}|\nabla(u_n- u)|+ |\nabla u|^{N-2}|\nabla(u_n- u)|+ |\nabla(u_n- u)|^{N-1}  \right)|\nabla v|.
			\end{aligned}
		\end{equation}	
		Moreover, by using H\"older inequality, we have
		\begin{equation}\label{ineq3}
			\begin{aligned}
				\int_{\mathbb{R}^N} |\nabla g_{N,s}|^{N-2}|\nabla(u_n- u)||\nabla v| \, dx \leq& \left(\int_{\mathbb{R}^N} |\nabla g_{N,s}|^{N-2}|\nabla(u_n- u)|^2 \, dx \right)^{\frac{1}{2}}
				\\&\times \left(\int_{\mathbb{R}^N} |\nabla g_{N,s}|^{N-2}|\nabla v|^2 \, dx\right)^{\frac{1}{2}},
			\end{aligned}
		\end{equation}
		\begin{equation}\label{ineq4}
			\begin{aligned}
				\int_{\mathbb{R}^N} |\nabla u|^{N-2}|\nabla(u_n- u)||\nabla v| \, dx \leq& \left(\int_{\mathbb{R}^N}|\nabla u|^N \, dx \right)^{\frac{N-2}{N}} \left(\int_{\mathbb{R}^N}|\nabla (u_n-u)|^N \, dx \right)^{\frac{1}{N}}\\&\times \left(\int_{\mathbb{R}^N}|\nabla v|^N \, dx \right)^{\frac{1}{N}},
			\end{aligned}
		\end{equation}
		and
		\begin{align}\label{ineq2}
			\int_{\mathbb{R}^N} |\nabla(u_n- u)|^{N-1}|\nabla v| \, dx &\leq \left(\int_{\mathbb{R}^N} |\nabla(u_n- u)|^N \, dx \right)^{\frac{N-1}{N}}\left(\int_{\mathbb{R}^N} |\nabla v|^N \, dx \right)^{\frac{1}{N}}.
		\end{align} 
		As a consequence of \eqref{ineq1}, \eqref{ineq3}, \eqref{ineq4}, \eqref{ineq2} and the convergence of $u_n$ in $W_{\mu_{N,s}}(\mathbb{R}^N)$, for $ \| v\|_{\mu_{N,s}}\leq 1$, we have 
		\begin{align*}
			\left| \langle P'_{N,s}(u) -  P'_{N,s}(u_n) , v \rangle \right|\leq C_N(1+\|\nabla u\|_{N}^{N-2} ) \|u_n -u\|_{\mu_{N,s}} + \|u_n -u\|_{\mu_{N,s}}^{N-2} \rightarrow 0,
		\end{align*}
		as $n\to +\infty$.
		Then, $P_{N,s}$ is Fr\'echet differentiable.			
		\end{proof}
	
	\begin{proposition}\label{DiffExponential}
		The functional $\bar{E}_{N,s} : W_{\mu_{N,s}}(\mathbb{R}^N) \to \mathbb{R}$ defined as
		$$ \dis \bar{E}_{N,s}(u):=  \ln \left( \int_{\mathbb{R}^N} e^{u} \, d\mu_{N,s}   \right) $$
		
		\noindent
		is of class $C^1$ and for any $u,v \in \wmns$ we have
		$$ \dis  \langle \bar{E}_{N,s}'(u),  v \rangle = \frac{ \dis \int_{\mathbb{R}^N}e^u v \, d\mu_{N,s}}{ \dis \int_{\mathbb{R}^N}e^u \, d\mu_{N,s}}.$$
	\end{proposition}
	\begin{proof}
		It suffices to show that the functional $E_{N,s} : W_{\mu_{N,s}}(\mathbb{R}^N) \to \mathbb{R}$ defined as
		$$ \dis E_{N,s}(u):=  \int_{\mathbb{R}^N} e^{u} \, d\mu_{N,s}  $$
		
		\noindent
		is of class $C^1$ and for any $u,v \in \wmns$ we have
		$$ \dis  \langle E_{N,s}'(u),  v \rangle = \int_{\mathbb{R}^N}e^u v \, d\mu_{N,s}.$$
		
		\noindent
		Let us prove that $E_{N,s}$ is Gâteaux differentiable.
		First, we show that for any $u,v \in  W_{\mu_{N,s}}(\mathbb{R}^N)$,
		\begin{align*}
			\langle E'_{N,s}(u), v \rangle=\left.\frac{d}{dt} E(u+tv)\right|_{t=0} = \int_{\mathbb{R}^N}e^u v \, d\mu_{N,s}.
		\end{align*} 
		
		\noindent
		Since the function $t \mapsto e^{u+tv}\mu_{N,s}$ is $C^1$, then
		\[\left.\frac{d}{d t} e^{u+tv}\mu_{N,s}\right|_{t=0}=e^u v \, \mu_{N,s}.\]
		
		Moreover, due to Theorem \ref{inclusionLq}, by H\"older inequality $e^uv\mu_{N,s}\in L^1(\mathbb{R}^N,dx)$. In particular,
		\begin{equation}\label{BoundExpNorm}
			\begin{aligned}
				\int_{\mathbb{R}^N}e^u v \, d\mu_{N,s} =&	\int_{\mathbb{R}^N}e^u \mu_{N,s}^{1/2} v \mu_{N,s}^{1/2} \, dx\leq \left(\int_{\mathbb{R}^N} e^{2u}\, d\mu_{N,s}\right)^{1/2}\left(\int_{\mathbb{R}^N} v^2\, d\mu_{N,s}\right)^{1/2}\\&\leq C_{\|u\|_{\mu_{N,s}}} \|v\|_{\mu_{N,s}},
			\end{aligned}
		\end{equation}
		so the directional derivative of $E_{N,s}$, is well defined as 
		\begin{align*}
			\langle E'_{N,s}(u), v \rangle= \int_{\mathbb{R}^N}e^u v \, d\mu_{N,s}.
		\end{align*} 
		
		Furthermore, for any $u \in W_{\mu_{N,s}}(\mathbb{R}^N)$ the map $v\rightarrow \langle E'_{N,s}(u), v\rangle$ is linear and, due to \eqref{BoundExpNorm}, it is also continuous in $W_{\mu_{N,s}}(\mathbb{R}^N)$. So $E_{N,s}$ is Gâteaux differentiable on $\wmns$. 
		Regarding the Fr\'echet differentiability, consider  $u_n\in W_{\mu_{N,s}}(\mathbb{R}^N)$ such that $u_n\rightarrow u$ in $ W_{\mu_{N,s}}(\mathbb{R}^N)$, then owing to the mean value theorem, we have
		\begin{align*}
			\left|\langle E'_{N,s}(u_n)-E'_{N,s}(u), v \rangle\right|=&\left|\int_{\mathbb{R}^N} (e^{u_n}-e^u)v \, d\mu_{N,s} \right|\leq 
			\int_{\mathbb{R}^N} |e^{u_n}-e^u||v| \mu_{N,s}\, dx\\&
			\leq \int_{\mathbb{R}^N} e^{\max{\{u_n,u\}}}\mu_{N,s}^{\frac{1}{3}}|u_n-u|\mu_{N,s}^{\frac{1}{3}}|v| \mu_{N,s}^{\frac{1}{3}}\, dx
			\\&\leq \left(\int_{\mathbb{R}^N} e^{3\max{\{u_n,u\}} }\mu_{N,s}\, dx\right)^{\frac{1}{3}} \left(\int_{\mathbb{R}^N} |v|^3\mu_{N,s}\, dx\right)^{\frac{1}{3}}
			\\&\quad \times  \left(\int_{\mathbb{R}^N} |u_n-u|^3\mu_{N,s}\, dx\right)^{\frac{1}{3}},
		\end{align*}
		where in the last inequality we have taken into account the Hölder extended inequality with exponents $3$,  $3$ and $3$.
		Let us analyze each term of the product.\par 
		\noindent
		Since $u_n,u\in W_{\mu_{N,s}}(\mathbb{R}^N)$, we can apply Theorem \ref{densitystatementANDextension} to the first term of the product and due to the fact $u_n\rightarrow u$ in $W_{\mu_{N,s}}(\mathbb{R}^N)$, there exists some $C>0$ independent of $n$ such that 
		\begin{align}\label{Bound1}
			\int_{\mathbb{R}^N} e^{3\max{\{u_n,u\}} }\mu_{N,s}\, dx\leq\int_{\mathbb{R}^N} \left(e^{3u_n}+e^{3u}\right)\mu_{N,s}\, dx< C.
		\end{align}
		
		Moreover, from Lemma \ref{inclusionLq}, we deduce that the second term of the product is bounded by some quantity independent of $n$ and $v$, since $\|v\|_{W_{\mu_{N,s}}(\mathbb{R}^N)}\leq 1$,
		\begin{align}\label{Bound2}
			\int_{\mathbb{R}^N} |v|^3\mu_{N,s}\, dx\leq  e^{C_N\|v\|^a_{\mu_{N,s}}}\leq e^{C_N}.
		\end{align}
		
		Finally, using again Lemma \ref{Lemma 2.1},
		\begin{align}\label{convergence}
			\int_{\mathbb{R}^N} |u_n-u|^3\mu_{N,s}\, dx\rightarrow 0.
		\end{align}
		
		As a consequence, from \eqref{Bound1}, \eqref{Bound2} and \eqref{convergence},
		\begin{align*}
			\left|\langle E'_{N,s}(u_n)-E'_{N,s}(u), v \rangle\right|\leq  C e^{C_N} \left(\int_{\mathbb{R}^N} |u_n-u|^3\mu_{N,s}\, dx\right)^{\frac{1}{3}}\rightarrow 0,
		\end{align*}
		so Fr\'echet differentiability is obtained.
	\end{proof}
	
	\begin{proof}[Proof of Theorem \ref{Diff}]
		
		The proof is deduced from Propositions \ref{DiffHNs} and \ref{DiffExponential}, taking into account that the functional $K_{N,s} : W_{\mu_{N,s}}(\mathbb{R}^N) \to \mathbb{R}$ defined as
		$$ \dis K_{N,s}(u):= \int_{\mathbb{R}^N} u \, d\mu_{N,s} $$
		\noindent
		is of class $C^1$, so for any $u,v \in W_{\mu_{N,s}}(\mathbb{R}^N)$,
		$$ \langle K_{N,s}'(u),  v \rangle = \int_{\mathbb{R}^N} v \, d\mu_{N,s}.$$
	\end{proof}

	\noindent
	{\bf Acknowledgments.}\\
	The first, second and fourth authors are supported by INdAM-GNAMPA.
	The first and the fourth authors are supported by the INdAM-GNAMPA projects \textsl{``Metodi variazionali e topologici tra Fisica, Geometria e Scienze Applicate''} and \textsl{``Strutture Analitiche e Geometriche in PDEs: Regolarità, Fenomeni Critici e Dinamiche Complesse''}  (CUP E53C25002010001).
	The third author has been supported by Junta de Andaluc\'ia FQM-194, FQM-424 (Spain), Plan Propio de Investigaci\'on y Transferencia de la Universidad de Almer\'ia (2023), and grant PID2025-173281NB-I00, funded by MICIU/AEI/10.13039/501100011033 and ESF+.


\begin{thebibliography}{999}	
		

		
		\bibitem{SanAdi-2007} Adimurthi and  K. Sandeep, \textit{A singular Moser–Trudinger embedding and its applications}, NoDEA Nonlinear Differential Equations Appl. \textbf{13} (2007), 585--603.
		
		
		\bibitem{Bec-1993} W. Beckner, \textit{Sharp Sobolev inequalities on the sphere and the Moser–Trudinger inequality}, Ann. of Math. \textbf{138} (1993), 213--242.
		
		\bibitem{BCMsurvey} N. Borgia, S. Cingolani, and G. Mancini, \textit{New Analytical and Geometrical Aspects on Trudinger-Moser Type Inequality in 2D}, in: Singularities, Asymptotics, and Limiting Models, Springer INdAM Ser., vol. 64, Springer, Singapore, 2025, 235--254. 
		
		\bibitem{BorCinMan-2025} N. Borgia, S. Cingolani, and G. Mancini, \textit{On the equivalence between an Onofri-type inequality by Del Pino–Dolbeault and the sharp logarithmic Moser–Trudinger inequality}, Calc. Var. Partial Differential Equations \textbf{64} (2025), Paper No. 80, 22 pp.
		
		\bibitem{CALVCORR} V. Calvez and L. Corrias, \textit{The parabolic-parabolic Keller-Segel model in $\mathbb{R}^2$}, Commun. Math. Sci. \textbf{6} (2008), no. 2, 417--447.
		
		\bibitem{CarCha-1986} L. Carleson and A. Chang, \textit{On the existence of an extremal function for an inequality of J. Moser}, Bull. Sci. Math. \textbf{110} (1986), no. 2, 113--127.
		
		\bibitem{CirEspLi-2026} G. Ciraolo, P. Esposito, and X. Li, \textit{On the classification of solutions to a class of $N$-Liouville equations in $\mathbb{R}^N$}, preprint arXiv:2604.10050 (2026).
		
		\bibitem{CsaNguRoy-2021} G. Csat\'o, V. H. Nguyen, and P. Roy, \textit{Extremals for the singular Moser-Trudinger inequality via n-harmonic transplantation}, J. Differential Equations \textbf{270} (2021), 843--882.
		
		\bibitem{dPinDol-2013} M. Del Pino and J. Dolbeault, \textit{The Euclidean Onofri inequality in higher dimensions}, Int. Math. Res. Not. IMRN (2013), no. 15, 3600--3611.
		
		\bibitem{DolEstJan-2015} J. Dolbeault, M. J. Esteban, and G. Jankowiak, \textit{The Moser-Trudinger-Onofri inequality}, Chinese Ann. Math. Ser. B \textbf{36} (2015), 777--802.
		
		\bibitem{DET} J. Dolbeault, M. J. Esteban, and G. Tarantello, \textit{The role of Onofri type inequalities in the symmetry properties of extremals for Caffarelli-Kohn-Nirenberg inequalities, in two space dimensions}, Ann. Sc. Norm. Super. Pisa Cl. Sci. (5) \textbf{7} (2008), no. 2, 313--341.
		
		\bibitem{DonLamLu-2018} M. Dong, N. Lam, and G. Lu, \textit{Sharp weighted Trudinger–Moser and Caffarelli–Kohn–Nirenberg inequalities and their extremal functions}, Nonlinear Anal. \textbf{173} (2018), 75--98.
		
		\bibitem{Esp-2018} P. Esposito, \textit{A classification result for the quasi-linear Liouville equation}, Ann. Inst. H. Poincar\'e Anal. Non Lin\'eaire \textbf{35} (2018), 781--801.
		
		\bibitem{Esp-2021} P. Esposito, \textit{Isolated singularities for the n-Liouville equation}, Calc. Var. Partial Differential Equations \textbf{60} (2021), Paper No. 137, 17 pp.
		
		\bibitem{GAJZAK} H. Gajewski and K. Zacharias, \textit{Global behaviour of a reaction-diffusion system modelling chemotaxis}, Math. Nachr. \textbf{195} (1998), 77--114.

		
\bibitem{IulManRoy-2017} S. Iula,  G. Mancini, \textit{Extremal functions for singular Moser–Trudinger embeddings}, Nonlinear Analysis \textbf{156}, (2017), 215-248.
		
		\bibitem{Kes-2006} S. Kesavan, \textit{Symmetrization \& Applications}, Series in Analysis, vol. 3, World Scientific Publishing Co. Pte. Ltd., Hackensack, NJ, 2006.
		
		\bibitem{KS} D. Kinderlehrer and G. Stampacchia, \textit{An introduction to variational inequalities and their applications}, Society for Industrial and Applied Mathematics, 2000.
		
		\bibitem{LamLu-2017} N. Lam and G. Lu, \textit{Sharp constants and optimizers for a class of Caffarelli-Kohn-Nirenberg inequalities}, Adv. Nonlinear Stud. \textbf{17} (2017), no. 3, 457--480.
		
		\bibitem{LamLu-2018} N. Lam and G. Lu, \textit{Weighted Moser–Onofri–Beckner and Logarithmic Sobolev Inequalities}, J. Geom. Anal. \textbf{28} (2018), no. 2, 1687--1715.
		
		
		\bibitem{LiebLoss} E. H. Lieb and M. Loss, \textit{Analysis}, Graduate Studies in Mathematics, vol. 14, American Mathematical Society, Providence, RI, second ed., 2001.
		
		
		\bibitem{Mos-1971} J. Moser, \textit{A sharp form of an inequality by N. Trudinger}, Indiana Univ. Math. J. \textbf{20} (1970/71), no. 11, 1077--1092.
		
		\bibitem{Mos-1973} J. Moser, \textit{On a nonlinear problem in differential geometry}, in: Dynamical Systems (M. M. Peixoto, ed.), Academic Press, New York, 1973, 273--280.
		
		\bibitem{Ono-1982} E. Onofri, \textit{On the positivity of the effective action in a theory of random surfaces}, Comm. Math. Phys. \textbf{86} (1982), no. 3, 321--326.
		
		\bibitem{PraTar-2001} J. Prajapat and G. Tarantello, \textit{On a class of elliptic problems in $\mathbb{R}^2$: symmetry and uniqueness results}, Proc. Roy. Soc. Edinburgh Sect. A \textbf{131} (2001), 967--985.

		
		\bibitem{Tru} N. S. Trudinger, \textit{On imbeddings into Orlicz spaces and some applications}, J. Math. Mech. \textbf{17} (1967), no. 5, 473--483.	
		
		
		
	\end{thebibliography}
\end{document}